\documentclass{article} 
\usepackage{amsmath,amssymb,amsthm} 
\usepackage[all]{xy}
\usepackage[OT2,T1]{fontenc}
\usepackage[utf8]{inputenc}
\usepackage{enumitem}
\usepackage{graphicx}
\usepackage{geometry}
\usepackage{hyperref}
\usepackage{tikz}
\usetikzlibrary{cd}
\usepackage{tikz-cd}
\usepackage{xcolor}
\usepackage{ stmaryrd }
\usepackage{mathrsfs}
\usepackage{titling}
\usepackage{fullpage}
\usepackage{mathtools}
\usepackage[english]{babel}
\usepackage{setspace}
\usepackage[normalem]{ulem}
\usepackage{mathdots}
\usepackage{microtype}

\usepackage{arydshln}

\usetikzlibrary{arrows,positioning,decorations.pathmorphing,
	decorations.markings
}
\tikzset{inner sep=0pt,
	root/.style={circle,draw,minimum size=7pt,thick},
	fatroot/.style={circle,draw,minimum size=10pt,thick},
	short root/.style={circle,fill,minimum size=7pt},
	doublearrow/.style={postaction={decorate},
		decoration={markings,mark=at position .7
			with {\arrow{angle 60}}},double distance=3pt,thick}
}

\DeclareUnicodeCharacter{00A0}{ }

\newtheorem{proposition}{Proposition}[section]
\newtheorem{definition}[proposition]{Definition}

\newtheorem{theorem}[proposition]{Theorem}
\newtheorem{lemma}[proposition]{Lemma}
\newtheorem{corollary}[proposition]{Corollary}

\newtheorem{construction}[proposition]{Construction}

\numberwithin{equation}{section}

\DeclareMathOperator{\GL}{GL}

\DeclareMathOperator{\SL}{SL}

\DeclareMathOperator{\SO}{SO}

\DeclareMathOperator{\Id}{Id}

\DeclareMathOperator{\Hom}{Hom}
\DeclareMathOperator{\Ext}{Ext}
\DeclareMathOperator{\Mat}{Mat}

\DeclareMathOperator{\Sym}{Sym}
\DeclareMathOperator{\vol}{vol}
\DeclareMathOperator{\Vol}{vol}

\DeclareMathOperator{\tr}{tr}

\DeclareMathOperator{\Stab}{Stab}

\DeclareMathOperator{\Spec}{Spec}

\renewcommand{\P}{\mathbb{P}}
\newcommand{\bbP}{\mathbb{P}}

\newcommand{\cO}{\mathcal{O}}

\newcommand{\cM}{\mathcal{M}}

\newcommand{\R}{\mathbb{R}}
\newcommand{\Q}{\mathbb{Q}}

\newcommand{\Z}{\mathbb{Z}}

\DeclareSymbolFont{cyrletters}{OT2}{wncyr}{m}{n}
\DeclareMathSymbol{\Sha}{\mathalpha}{cyrletters}{"58}

\newcommand{\extp}{\@ifnextchar^\@extp{\@extp^{\,}}}
\def\@extp^#1{\mathop{\bigwedge\nolimits^{\!#1}}}

\DeclareUnicodeCharacter{00A0}{ }
\newcommand{\height}{\mathrm{Ht}}

\DeclareMathOperator{\disc}{disc}

\newcommand{\Siegel}{\mathfrak{S}}

\newcommand{\ShHom}{\mathcal{H}\!\mathit{om}}

\title{Lower bounds on heights of odd degree points of hyperelliptic curves}
\author{Jef Laga and Jack A. Thorne}

\begin{document}

\maketitle

\begin{abstract}
    We develop a reduction theory for the representation of $\SL_n$ on pairs of symmetric $n\times n$ matrices.
    We apply this theory to the pencils of quadrics arising from divisors on hyperelliptic curves.
    We use these results to show that, in a density $1$ family, an odd degree point $P$ of degree at most $2g-1$ on the hyperelliptic curve $z^2 = f_0x^{2g+2} + f_1 x^{2g+1} y + \cdots + f_{2g+2}y^{2g+2}$ cannot have small Weil height.
\end{abstract}

\tableofcontents

\section{Introduction}

\subsection{Results}
Let $g\geq 1$ be an integer.
Consider the set of integral binary forms $f(x,y) = f_0x^{2g+2}  +f_1 x^{2g+1}y + \cdots + f_{2g+2} y^{2g+2} \in \Z[x,y]$ of degree $2g+2$ and nonzero discriminant, ordered by the height $\height(f) = \max |f_i|$. 
Such an $f$ determines a hyperelliptic genus-$g$ curve $X_f$ over $\Q$ with equation 
\begin{align}\label{eq_introhyperellipticcurve}
    X_f \colon z^2 = f(x,y) 
\end{align}
in weighted projective space $\P(1,1,g+1)$.
When ordered by height, a positive proportion of $X_f$ are everywhere locally soluble, i.e., have $\R$-points and $\Q_p$-points for all primes $p$ \cite{PoonenStoll-localglobaldensities}.

If $g=1$, Bhargava has shown that among the everywhere locally soluble $X_f$, a positive proportion have a rational point, and a positive proportion have not \cite{Bhargava-hasseprincipleplanecubics}.
When $X_f$ does have a rational point, it is natural to ask: how complicated must such a point be?
In this paper, we show that the heights of such points are, typically, not too small: 
\begin{theorem}\label{theorem_intro1}
    Let $\epsilon>0$ be arbitrary and $g=1$.
    Then for $100\%$ of integral binary quartic forms $f$ (ordered by height), every rational point $P = (x:y:z)\in X_f(\Q)$ satisfies 
    \begin{align}
        h((x:y)) \geq \left(\frac{5}{4}- \epsilon\right) \log \height(f),
    \end{align}
    where $h(\alpha)$ denotes the logarithmic Weil height of an element $\alpha \in \P^1(\bar{\Q})$.
\end{theorem}

We prove a similar theorem for hyperelliptic curves of every genus $g\geq 1$.
Bhargava--Gross--Wang have shown that among the everywhere locally soluble $X_f$, a positive proportion have no points over any odd degree extension of $\Q$ \cite{BGW17}.
Analogously to the $g=1$ case, it seems reasonable to expect (using the heuristics of Poonen and Rains \cite{PoonenRainsRandom}) that a positive proportion of $X_f$ do have an algebraic point of odd degree at most $g+1$, and again we can ask about the complexity of such a point.

\begin{theorem}\label{theorem_intro2}
    Let $\epsilon >0$ be arbitrary.
    Then for $100\%$ of integral binary forms $f$ of degree $2g+2$ and nonzero discriminant, every algebraic point $P  = (x:y:z) \in X_f(\bar{\Q})$ of odd degree $m\leq 2g-1$ satisfies
    \begin{align}
        m\cdot  h((x:y)) \geq \left( 1+ \frac{1}{2g+2} - \epsilon \right) \log \height(f).
    \end{align}
\end{theorem}
In fact, we prove a slightly more general result that applies to effective divisors of degree $2g-1$ on $X_f$, see Theorem \ref{theorem_lowerboundheightdivisor}. 

Theorems \ref{theorem_intro1} and \ref{theorem_intro2} are analogous to results proved in our previous work \cite{lagathorne2024smallheightoddhyperelliptic} for the family of odd hyperelliptic curves of genus $g$ (roughly speaking, the subfamily defined by the conditions $f_0 = 0$, $f_1 = 1$). We compare the methods used in this work to those of \cite{lagathorne2024smallheightoddhyperelliptic} below.

\subsection{Methods}
Let $n=2g+2$ and consider the representation of $\SL_n$ on the space $V$ of pairs of symmetric $n\times n$ matrices, where the action is defined by $g\cdot (A,B) = (g^{-t}Ag^{-1},g^{-t}Bg^{-1})$.
To such a pair $(A,B)$, we may associate an invariant binary form 
\begin{align*}
    f_{A,B}(x,y) = (-1)^{n(n-1)/2}\det(Ax -By)
\end{align*}
whose coefficients freely generate the ring of polynomial invariants for the $\SL_n$-action on $V$ \cite{BGW17}.

A key step in our proof of Theorem \ref{theorem_intro2} is the development of a reduction theory for the $\SL_n(\Z)$-action on $V(\Z)$.
Generally speaking, given an action of an arithmetic group on a set, reduction theory aims to find `reduced' representatives of orbits for the action, e.g., representatives whose coefficients are small.
In our setting, we define for every pair $(A,B)\in V(\R)$ with $\disc(f_{A,B})\neq 0$ an inner product $H_{A,B}$ on $\R^n$, whose formation is $\SL_n(\R)$-equivariant.
We say that a pair $(A,B)$ is Minkowski (or LLL) reduced if the standard basis of $\R^n$ is Minkowski (or LLL) reduced with respect to $H_{A,B}$.
The coefficients of reduced pairs (in either sense) can be absolutely bounded in terms of the coefficients of $f_{A,B}$, and so the problem of reducing elements in $V(\Z)$ becomes a problem in the reduction theory of lattices, which is well understood.

Let $X = \SL_n(\R)/\SO_n(\R)$ denote the symmetric space of inner products on $\R^n$ up to scaling. 
We call the association $(A,B)\mapsto H_{A,B}$ or, by passage to the quotient, the association
\begin{align}\label{eqintro_reductioncovariant}
    \mathcal{R}\colon \SL_n(\Z) \backslash V(\Z)^{\Delta \neq 0} \rightarrow \SL_n(\Z)\backslash X
\end{align}
the `reduction covariant' of the representation $V$, akin to the reduction covariant defined in \cite{CremonaFisherStoll-minimisationreductioncoveringselliptic, StollCremona-reductionbinaryforms, Thorne-reduction}.
In the context of \cite{CremonaFisherStoll-minimisationreductioncoveringselliptic}, the reduction covariant is used in the algorithmic study of $n$-descent on elliptic curves, and we similarly expect our reduction covariant to be useful for computational purposes.

In this paper, we instead employ the reduction covariant for theoretical purposes, and relate it to the descent theory of hyperelliptic curves.
Suppose $f(x,y)\in \Z[x,y]$ is a binary form of degree $n$ and nonzero discriminant, and let $V_f$ be the subset of $(A,B)$ with $f_{A,B} = f$.
Given an effective $\Q$-rational divisor $D$ on $X_f$ of degree $2g-1$, we build on the methods of \cite{BGW17} to produce an integral orbit $\SL_n(\Z)\cdot (A,B) \subset V_f(\Z)$, which then has an associated reduction covariant $[H_{A, B}] \in \SL_n(\Z) \backslash X$. 
We show that (in a precise sense) if $D$ has small height, then $\Z^n$ has a primitive element of small norm with respect to $H_{A,B}$.
On the other hand, we show that the reduction covariants of orbits of height bounded by $X$  are equidistributed as $X \to \infty$, with respect to the natural measure on the target of (\ref{eqintro_reductioncovariant}). 
Since the subset of inner products in $\SL_n(\Z)\backslash X$ admitting a nonzero vector of small norm has small measure, we conclude that only a small proportion of curves $X_f$ can admit a divisor $D$ of small height.

The construction of integral orbits alluded to in the previous paragraph is a key technical innovation in this paper, which we achieve by introducing a new method to construct such orbits.
This is essentially a local question, so we may begin with a rational odd degree divisor over $\Q_p$ and construct an element of $V(\Z_p)$. 
Such a construction was given in our setting in \cite[Theorem 15]{BGW17}, although their argument contains gaps, which were corrected by Swaminathan \cite[Appendix 4A]{Swaminathan-thesis}.
A construction in the setting of odd hyperelliptic curves was given by Bhargava--Gross \cite[Propostion 19]{BhargavaGross}, but this includes a reduction step that makes it not useful for our purpose here (where we have to be able to show that a specific vector is integral with respect to the given integral structure, in order to know that its length cannot be too small too often). In our previous work \cite{lagathorne2024smallheightoddhyperelliptic}, we resolved this by a more careful construction that still, like the one in \cite{BhargavaGross}, relies on an induction of the degree of the divisor. 

The construction of integral orbits we give here uses the interpretation of these orbits given by Wood \cite{Woo14}, namely as corresponding to certain classes of ideals for the ring associated to the binary form $f(x, y)$. This ring may be thought of as the global sections of the structure sheaf on the subscheme $S_f$ of $\P^1_{\Z_p}$ cut out by $f$. We associate to a divisor $D$ a torsion-free sheaf $\cO_X(D)$ on the hyperelliptic curve $X_f$ over $\Z_p$ defined by the same equation $z^2 = f(x, y)$; this curve contains $S_f$ as the locus $z = 0$, and the integral orbit we construct is the one associated to the ideal class of $H^0(S_f, \cO_X(D)|_{S_f})$. The hard part is to show that this ideal class has the right properties to correspond to an integral orbit; this we accomplish using coherent Grothendieck duality for the scheme $S_f$. 

We now describe the relation between the methods employed here and those used to prove the main theorems of our previous work on odd hyperelliptic curves \cite{lagathorne2024smallheightoddhyperelliptic}. In that paper, we employed a broadly similar strategy, using the equidistribution of reduction covariants, to obtain height lower bounds in a density-$1$ family of monic odd genus-$g$ hyperelliptic curves, using the representation of $\SO_{2g+1}$ on self-adjoint traceless $(2g+1)\times (2g+1)$ matrices. However, there are some key differences. 
First, the representation of $\SL_n$ considered in this paper falls outside the scope of Vinberg theory, so the general construction of the reduction covariant of \cite{Thorne-reduction} does not apply here. In fact, for the representation of $\SL_n$ there are infinitely many reduction covariants, and it is important to select one with the correct properties. 
Second, the reduction covariant here equidistributes over all integral orbits of nonzero discriminant in $V$, while in \cite{lagathorne2024smallheightoddhyperelliptic} it only equidistributes over the \emph{irreducible} integral orbits, and handling the reducible orbits required an additional argument using the methods of \cite{Bhargava-squarefree}. 
Finally, the construction of integral orbits described above uses a different method to the one employed in \cite{lagathorne2024smallheightoddhyperelliptic}. It seems likely that the method used here would also work in that setting. 

\subsection{Organization}

In \S\ref{sec_binary_forms}, we recall some results of Wood \cite{Woo11, Woo14} concerning binary forms and the rank $n$ rings they define.
We emphasize the role played by the closed subscheme $S_f\subset \P^1_{\Z}$ cut out by the form $f$ and show, using coherent duality, how certain sheaves on $\mathcal{M}$ give rise to integral orbits for the representation $V$.
In \S\ref{sec_the_representation} we describe the representation $V$ of $\SL_n$ together with its rational and integral orbits.
We also introduce the reduction covariant of an element $(A,B) \in V(\R)$ of nonzero discriminant.
In \S\ref{sec_integral_orbit_representatives}, we present the crucial construction of rational and integral orbits associated to effective divisors of odd degree on $X_f$ satisfying some conditions.
Moreover, we compute the norm of a certain element in $\Z^n$ with respect to the reduction covariant of this integral orbit.
In \S\ref{sec: equidistribution}, we prove that the reduction covariant \eqref{eqintro_reductioncovariant} equidistributes, by adapting the geometry-of-numbers methods of Bhargava \cite{bhargava2015mosthyperellipticarepointless}.
In \S\ref{sec: a height lower bound for divisors}, we combine all of these ingredients to prove the main theorems of the introduction.

\section{Binary forms and ideal classes}\label{sec_binary_forms}
There is a long-studied association between binary forms of fixed degree $n$ and rings of rank $n$ (see e.g.\ \cite{Nak89}), which has been analysed in detail by Wood \cite{Woo11}, who also studied the relation between ideal classes for these rings and orbits in the representation $V$ of $\SL_n$ studied later in this paper \cite{Woo14}. In this section, we first recall some of Wood's results, and complement them with a geometric construction of orbits associated to sheaves on the zero locus of $S_f$. 
Our desire to include forms that are not primitive makes various proofs more technical.

Let $n \geq 2$ be an integer, and let $A$ be a Dedekind domain of fraction field $K$ of characteristic $0$.
If $k \in \mathbb{Z}$, write $\cO_{\bbP^1_A}(k)$ for the usual sheaf on $\bbP^1_A$ whose global sections are the forms in $A[x, y]$ that are homogeneous of degree $k$.
Let
\[ f(x, y) = f_0 x^n + f_1 x^{n-1} y + \dots + f_n y^n \in H^0(\bbP^1_A, \cO_{\bbP^1_A}(n)) \]
be  a form whose discriminant $\Delta(f) \in A$ is nonzero, define $S_f \subset \mathbb{P}^1_A$ to be the closed subscheme defined by the vanishing of $f$, and $R_f = H^0(S_f, \cO_{S_f})$.
\begin{proposition}\label{prop_R_f_is_free_over_A}
    The ring $R_f$ is a free $A$-algebra of rank $n$. 
\end{proposition}
\begin{proof}
    We consider the short exact sequence of coherent sheaves on $\bbP^1_A$:
    \[ 0 \to \cO_{\bbP^1_A}(-n) \overset{\times f}{\to} \cO_{\bbP^1_A} \to \cO_{S_f} \to 0. \]
    After passing to global sections, we get a short exact sequence
    \[ 0 \to A \to R_f \to H^1(\bbP^1_A, \cO_{\bbP^1_A}(-n)) \to 0. \]
    Since $H^1(\bbP^1_A, \cO_{\bbP^1_A}(-n))$ is a free $A$-module of rank $n-1$, this completes the proof. 
\end{proof}
\begin{proposition}\label{prop_primitive_implies_weierstrass_locus_is_affine}
    If $f(x, y)$ is primitive, then the scheme $S_f$ is affine. 
\end{proposition}
To say that $f(x, y)$ is primitive is to say that its content (i.e. the ideal of $A$ generated by $f_0, \dots, f_n$) is the unit ideal. 
\begin{proof}
    It suffices to prove this affine locally on $\Spec A$, so we can assume that there exists a homogeneous form $g(x, y) \in A[x, y]$ of some degree $m$ such that for every non-zero prime ideal $P \leq A$, the reductions modulo $P$ $\overline{f}, \overline{g} \in (A / P)[x, y]$ are non-zero and cut out disjoint closed subschemes of $\bbP^1_{A / P}$. Then $S_f$ is identified with a closed subscheme of the distinguished open $D^+(g) \subset \bbP^1_A$, and is therefore affine.
\end{proof}
If $k \in \mathbb{Z}$, we write $\cO_{S_f}(k)$ for the pullback of $\cO_{\bbP^1_A}(k)$ to $S_f$. We define $I_f = H^0(S_f, \cO_{S_f}(n-3))$ and $J_f = H^0(S_f, \cO_{S_f}(n-2))$. Then $I_f, J_f$ are finite $R_f$-modules that are free over $A$ of rank $n$ (the proof is a twist of the proof of Proposition \ref{prop_R_f_is_free_over_A}). Taking cohomology, we see that there are short exact sequences
\begin{equation}\label{eqn_I_f} 0 \to H^0(\bbP^1_A, \cO_{\bbP^1_A}(n-3)) \to I_f \to H^1(\bbP^1_A, \cO_{\bbP^1_A}(-3)) \to 0 
\end{equation}
and
\begin{equation}\label{eqn_J_f} 0 \to  H^0(\bbP^1_A, \cO_{\bbP^1_A}(n-2)) \to J_f \to H^1(\bbP^1_A, \cO_{\bbP^1_A}(-2)) \to 0. 
\end{equation}
The group $H^1(\bbP^1_A, \cO_{\bbP^1_A}(-2))$ is free of rank 1 over $A$, a basis being given by the class represented in Cech cohomology with respect to the cover $D^+(x), D^+(y)$ by the element $(xy)^{-1} \in H^0(D^+(xy), \cO_{\bbP^1_A}(-2))$. We write $\zeta : J_f \to A$ for the $A$-linear form induced by the exact sequence (\ref{eqn_J_f}) and this choice of basis element. 

If $a(x, y) \in H^0(\bbP^1_A, \cO_{\bbP^1_A}(1))$ is a linear form, write $\mathrm{ev}_a : I_f \to A$ for the composite
\[\mathrm{ev}_a\colon  I_f \to H^1(\bbP^1_A, \cO_{\bbP^1_A}(-3)) \to H^1(\bbP^1_A, \cO_{\bbP^1_A}(-2)) \to A, \]
where the second map is induced by multiplication by $a$, and the third is $\zeta$.

\begin{lemma}\label{lemma_identificationconnectingmaps}
    For every linear form $a$, $\mathrm{ev}_a$ equals the composite
    \[\beta_a\colon  I_f \to J_f \to A, \]
    where the first map is induced by passage to global sections of the morphism $\times a : \cO_{S_f}(n-3) \to \cO_{S_f}(n-2)$, and the second map is $\zeta$.
\end{lemma}
\begin{proof}
     This follows by considering the commutative diagram with exact rows
    \[ \xymatrix{ 0 \ar[r] & \cO_{\bbP^1_A}(-3)\ar[d] \ar[r] & \cO_{\bbP^1_A}(n-3) \ar[d] \ar[r] & \cO_{S_f}(n-3) \ar[d] \ar[r] & 0 \\ 
    0 \ar[r] & \cO_{\bbP^1_A}(-2) \ar[r] & \cO_{\bbP^1_A}(n-2) \ar[r] & \cO_{S_f}(n-2)  \ar[r] & 0, }
    \]
    where the vertical maps are all multiplication by $a$. 
\end{proof}

\begin{proposition}\label{prop_application_of_global_duality}
    Let $\cM$ be a coherent sheaf on $S_f$ such that $H^1(S_f, \cM) = 0$ (this is automatic if $f$ is primitive), and let $M = H^0(S_f, \cM)$. Then there is a canonical isomorphism
    \[ \Hom_{\cO_{S_f}}(\cM, \cO_{S_f}(n-2)) \cong \Hom_A(M, A), \]
    given by passage to global sections and then composition with $\zeta$. 
\end{proposition}
\begin{proof}
    We will use coherent Grothendieck duality in the form of \cite{Har66, conrad-grothendieckduality}.
    For a Noetherian scheme $X$, denote by $\mathbf{D}(X)$ the derived category of $\cO_X$-modules and $\mathbf{D}^+_c(X)$ (respectively $\mathbf{D}^-_c(X)$) the full subcategory consisting of complexes whose cohomology sheaves are coherent and vanish in sufficiently negative degrees (respectively positive degrees).
    The theory associates, to any morphism $\varphi\colon X\rightarrow Y$ of finite type schemes over $A$, a functor $\varphi^!\colon \mathbf{D}^+_c(Y)\rightarrow \mathbf{D}^+_c(X)$ satisfying various properties explained in \cite[Section 3.3]{conrad-grothendieckduality}, and when $\varphi$ is furthermore proper, a trace map $\tr_{\varphi}\colon R\varphi_* \circ \varphi^{!} \Rightarrow \Id$, which is a natural transformation of endofunctors on $\mathbf{D}^+_c(Y)$ \cite[\S3.4]{conrad-grothendieckduality}.
    Grothendieck--Serre duality then states that for all $F \in \mathbf{D}_c^-(X)$ and $G\in \mathbf{D}_c^+(Y)$, the composition 
    \begin{align}\label{eq_grothserredualitymap}
        \Hom_{\mathbf{D}(X)}(F, \varphi^! G) \xrightarrow{R\varphi_*} 
        \Hom_{\mathbf{D}(Y)}(R\varphi_* F, R\varphi_* \varphi^! G) \xrightarrow{\tr_{\varphi}}
        \Hom_{\mathbf{D}(Y)}(R \varphi_* F, G)
    \end{align}
    is an isomorphism of $A$-modules; this follows from applying $H^0\circ R\Gamma$ to \cite[Theorem 3.4.4]{conrad-grothendieckduality}.

    In our concrete situation, we have proper maps $i\colon S_f\hookrightarrow \P^1_A$, $p\colon \P^1_A \rightarrow \Spec(A)$ and $\pi\colon S_f\rightarrow \Spec(A)$ satisfying $\pi = p \circ i$.
    We will show that Grothendieck duality for the triple $(\varphi, F,G) = (\pi, \cM, A)$ implies the claim of the proposition. 
    In our arguments we will sometimes identify quasi-coherent modules on $\Spec(A)$ with $A$-modules.
    We first make the theory explicit for $i$ and $p$. 
    \begin{itemize}
        \item If $\mathcal{F}$ is a coherent sheaf on $\P^1_A$, then $Ri_* i^! \mathcal{F} = R\ShHom_{\mathbf{D}(\P^1_A)}(i_*\cO_{S_f},\mathcal{F})$ and $\tr_i\colon R\ShHom_{\mathbf{D}(\P^1_A)}(i_*\cO_{S_f},\mathcal{F})\rightarrow \mathcal{F}$ is given by precomposing with the quotient map $\cO_{\P^1_A}\rightarrow i_*\cO_{S_f}$ and applying the canonical isomorphism $R\ShHom_{\mathbf{D}(\P^1_A)}(\cO_{\P^1_A}, \mathcal{F}) \simeq \mathcal{F}$; see \cite[Lemma 3.4.3(2) and page 31]{conrad-grothendieckduality}.
        Since $S_f$ is a Cartier divisor, $i_*\cO_{S_f}$ is computed by the complex of locally free sheaves 
        \[ [ \cO_{\P^1_A}(-n) \xrightarrow{\times f} \cO_{\P^1_A} ] \]
        (with $\cO_{\P^1_A}$ placed in degree $0$), so $R\ShHom_{\mathbf{D}(\P^1_A)}(i_*\cO_{S_f},\mathcal{F})$ is represented by the complex 
        \[ [ \mathcal{F} \xrightarrow{\times f} \mathcal{F}(n) ] \]
        living in degrees $0$ and $1$, and $\tr_i$ is induced by the map of complexes 
        \[ [ \mathcal{F} \rightarrow \mathcal{F}(n) ] \rightarrow [ \mathcal{F} \to 0 ] \] which is the identity in degree $0$ and the zero map in degree $1$. 
        If $\mathcal{F}$ is furthermore locally free, the map $\mathcal{F}\xrightarrow{\times f} \mathcal{F}(n)$ is injective, so there is a quasi-isomorphism
        \[  [ \mathcal{F} \rightarrow \mathcal{F}(n) ] \overset{\sim}{\to} i_* i^*(\mathcal{F}(n))[-1]. \]
        \item If $\mathcal{F}$ is a coherent sheaf on $\Spec(A)$, then $p^!(\mathcal{F}) = \Omega^1_{\P^1_A/A}[1] \otimes_{\cO_{\P^1_A}} f^*\mathcal{F}$ by \cite[Lemma 3.4.3(3)]{conrad-grothendieckduality}.
        Moreover if $\mathcal{F} = A$ then $\tr_p\colon Rp_* p^! A \rightarrow A$ equals, when viewed as a map of $A$-modules $H^1(\P^1_A,\Omega_{\P^1_A/A}^1)\rightarrow A$, the trace isomorphism from classical Serre duality fixed in \cite[Section 2.3]{conrad-grothendieckduality}. 
    \end{itemize}
    Using these examples, the isomorphism $\Omega_{\P^1_A / A}^1 \cong \cO_{\P^1_A}(-2)$, and the natural isomorphism $\pi^! \cong i^! \circ p^!$ \cite[Equation (3.3.14)]{conrad-grothendieckduality}, we compute
    \[ \pi^! A \cong i^! p^! A  \cong \cO_{S_f}(n-2), \]
    in $\mathbf{D}^+_c(S_f)$, and
    \[ R \iota_\ast \pi^! A \cong i_*\cO_{S_f}(n-2) \cong [ \cO_{\P^1_A}(-2) \overset{\times f}{\to} \cO_{\P^1_A}(n-2) ]. \]
    in $\mathbf{D}^+_c(\P^1_A)$. We observe that the short exact sequence (\ref{eqn_J_f}) defining the map $\zeta$ expresses the filtration associated to the quasi-isomorphism $\cO_{S_f}(n-2) \cong [ \cO_{\P^1_A}(-2) \overset{\times f}{\to} \cO_{\P^1_A}(n-2) ]$, the map $\zeta$ itself corresponding to the composite of the Serre duality isomorphism $H^1(\P^1_A, \cO_{\P^1_A}(-2)) \cong A$ with the map $H^0(S_f, \cO_{S_f}(n-2)) \to H^1(\P^1_A,\cO_{\P^1_A}(-2))$ determined by the composite 
    \[ i_*\cO_{S_f}(n-2) \cong [ \cO_{\P^1_A}(-2) \overset{\times f}{\to} \cO_{\P^1_A}(n-2) ] \to \cO_{\P^1_A}(-2)[1] \]
    in $\mathbf{D}^+_c(\P^1_A)$. 
    
    We now connect this to trace. Using the examples above, we see that $\tr_i(p^!A) : R i_\ast i^! p^! A \to p^! A$ is the morphism 
    \[ [ \cO_{\P^1_A}(-2) \overset{\times f}{\to} \cO_{\P^1_A}(n-2) ] \to \cO_{\P^1_A}(-2)[1], \]
    showing that $R p_\ast(\tr_i(p^!A)) : R p_\ast R i_\ast i^! p^! A \to R p_\ast p^! A$ is the induced morphism
    \[ J_f = H^0(\P^1_A, i_\ast \cO_{S_f}(n-2)) = H^0(\P^1_A, [ \cO_{\P^1_A}(-2) \overset{\times f}{\to} \cO_{\P^1_A}(n-2) ]) \to H^1(\P^1_A, \cO_{\P^1_A}(-2)).  \]
    On the other hand, by \cite[Lemma 3.4.3(1)]{conrad-grothendieckduality}, we have $\tr_{\pi} = \tr_p \circ Rp_*(\tr_i)$, so $\tr_\pi(A) : R \pi_\ast \pi^! A \to A$ is the morphism $J_f \to A$ obtained by composing this induced morphism with the Serre duality isomorphism $\tr_p(A) : H^1(\P^1_A, \cO_{\P^1_A}(-2)) \to A$. We have shown that $\tr_\pi(A) = \zeta$. (We note that the references \cite{Har66,conrad-grothendieckduality} contain, at points, different choices, that may affect the sign of $\tr_\pi$, and therefore the truth of the statement $\tr_\pi(A) = \zeta$; however, the truth of Proposition \ref{prop_application_of_global_duality}, namely that a certain map of $A$-modules induced by $\zeta$ is an isomorphism, is insensitive to multiplying $\zeta$ by a sign, so we do not need to worry about this here.)

    We now show how the desired statement follows from (\ref{eq_grothserredualitymap}). By assumption, we have $H^1(S_f, \cM) = 0$, hence $R\pi_\ast \cM \cong M$. We have shown that there is an isomorphism $\pi^! A \cong \cO_{S_f}(n-2)$. We find that $\zeta$ induces an isomorphism
    \[ \Hom_{\mathbf{D}(S_f)}(\cM, \cO_{S_f}(n-2)) \cong \Hom_{\mathbf{D}(A)}(M, A). \]
    Since the embedding of the category of $\cO_X$-modules in $\mathbf{D}(X)$ is fully faithful, this is exactly what we needed to show. 
\end{proof}
We can use a similar point of view to construct pencils of quadrics over $A$ from suitable sheaves on $S_f$. Suppose given a coherent sheaf $\cM$ on $S_f$ satisfying the following conditions:
\begin{itemize}
    \item $\cM_\eta$ is free of rank 1 at each generic point $\eta \in S_f$. 
    \item There is an isomorphism $\cM \cong \ShHom_{\cO_{S_f}}(\cM, \cO_{S_f}(n-3))$.
    \item $H^1(S_f, \cM) = 0$.
\end{itemize}
\begin{proposition}\label{prop_pencil_of_quadrics_over_A}
    With assumptions as above, let $M = H^0(S_f, \cM)$. Then:
    \begin{enumerate}
        \item $M$ is a finite $R_f$-module, $A$-torsion-free, and $M \otimes_A K$ is free of rank 1 over $R_f \otimes_A K$.
        \item By passage to global sections, we obtain an $A$-linear map $M \times M \to I_f \to H^0(\bbP^1_A, \cO_{\bbP^1}(1))^\vee$. Fix an $A$-basis of $M$, and let $A_x, A_y$ denote the matrices of the two symmetric bilinear forms $M \times M \to A$ given by composition with $\mathrm{ev}_x$ and $\mathrm{ev}_y$, respectively. Then there exists a unit $u \in A^\times$ such that $\det( x A_y - y A_x ) = u f(x, y)$. 
    \end{enumerate}
\end{proposition}
\begin{proof}
    We take each point in turn. Since $S_f$ is projective, $M$ is a finite $R_f$-module. Since $S_{f, K}$ is a finite \'etale $K$-scheme, and $\cM_K$ is assumed locally free of rank 1, $M \otimes_A K$ is free of rank 1 over $R_f \otimes_A K$. We must check that $M$ is $A$-torsion-free. We can check this locally on $A$, so can assume that $A$ is a DVR with uniformizer $\pi$. We are also free to make an \'etale localisation on $A$, so can assume that there is a primitive homogeneous linear polynomial $a(x, y) = x - s y \in A[ x, y]$ such that $f(s, 1) \neq 0$. Consider the short exact sequence of sheaves on $\bbP^1_A$
    \[ 0 \to \cO_{\bbP^1_A}(n-3) \overset{\times a}{\to} \cO_{\bbP^1_A}(n-2) \to \cO_a \to 0, \]
    where $\cO_a$ (viewed as a sheaf on $\P^1_A$) is the structure sheaf of the closed subscheme $V(a) \subset \P^1_A$, isomorphic to $\Spec A$. This short exact sequence remains exact after pullback along $i : S_f \to \P^1_A$ (because $A[x] / a(x, 1)$ has no $f(x, 1)$-torsion). After making this pullback and applying the functor $\Hom_{\cO_{S_f}}(\cM, -)$, we get an exact sequence
    \begin{equation}\label{eqn_short_exact_sequence} 0 \to \Hom_{\cO_{S_f}}(\cM, \cO_{S_f}(n-3)) \to \Hom_{\cO_{S_f}}(\cM, \cO_{S_f}(n-2)) \to \Hom_{\cO_{S_f}}(\cM, i^\ast \cO_a). 
    \end{equation}
    Applying
    our hypotheses on $\cM$ and Proposition \ref{prop_application_of_global_duality}, we see in particular that we have an embedding $M \to \Hom_A(M, A)$. The module $\Hom_A(M, A)$ is certainly $A$-torsion-free, so this shows that $M$ is itself $A$-torsion-free. 

    To prove the second part, we proceed in stages. Let $h(x, y) = \det( x A_y - y A_x ) \in A[x, y]$. We will first show that $h(x, y)$ and $f(x, y)$ cut out the same locus in $\bbP^1_K$, using a calculation over the algebraic closure. This implies that $f(x, y)$ and $h(x, y)$ are equal up to multiple by some element of $K^\times$. To show equality, we will then be free to localise on $A$ and assume that $A$ is a DVR of uniformizer $\pi$. We will then show the existence of a pair $(r, s ) \in A^2$ such that $h(r, s)$ and $f(r, s)$ have the same (finite) $\pi$-adic valuation, forcing $f(x, y)$ and $h(x, y)$ to be equal up to multiplication by units. 

    It follows from Lemma \ref{lemma_identificationconnectingmaps} that if $a(x, y) = r x + s y$, we write $\varphi : M \times M \to I_f$ for the $R_f$-bilinear pairing, and $A_a$ for the matrix of the symmetric $A$-bilinear form $\mathrm{ev}_a \circ \varphi : M \times M \to A$, then $A_a = r A_x + s A_y$, and hence $\det(A_a) = \det(s A_y + r A_x) = h(s, -r)$. 

    We now show that $f(x, y)$ divides $h(x, y)$. This can be checked after base extension to the algebraic closure of $K$. Suppose that $a(x, y)$ a linear form dividing $f(x, y)$. Then multiplication by $a$ sends $I_f$ to a codimension 1 subspace of $J_f$. Using Lemma \ref{lemma_identificationconnectingmaps}, we see that $\mathrm{ev}_a \circ \varphi = \zeta \circ \beta_a \circ \varphi$ does not have full rank, showing that $\det(A_a) = 0$ and hence $a(x, y)$ divides $h(x, y)$. Since $f(x, y)$ has non-zero discriminant and $a$ was an arbitrary linear factor, this shows that $f$ divides $h$; since they have the same degree, they are in fact equal up to an element of $K^{\times}$.

    Now we return to the case of general $A$, and must show that $f, h \in A[x, y]$ are in fact equal up to multiplication by elements of $A^\times$. After localisation, we can assume that $A$ is a DVR of uniformizer $\pi$ and with infinite residue field $k$. Let us write $f(x, y) = \pi^m g(x, y)$ where $m \geq 0$ and $g(x, y)$ is primitive. We first treat the case where $m = 0$, i.e.\ $f(x, y)$ is itself primitive. In this case,  we can choose a primitive linear form $a(x, y) = r x + s y \in A[x, y]$ such that $\overline{a}$ does not divide $\overline{f}$ in $k[x, y]$ (where overline denotes reduction modulo $\pi$); equivalently, $f(s, -r) \in A^\times$. Forming the exact sequence (\ref{eqn_short_exact_sequence}), we get a short exact sequence
    \[ 0 \to M \to \Hom_A(M, A) \to \Hom_{\cO_{S_f}}(\cM, i^\ast \cO_a)  \]
    The group $\Hom_{\cO_{S_f}}(\cM, i^\ast \cO_a)$ is in fact $0$, because our choice of $a$ means that $i^\ast \cO_a = 0$ (in other words, the section $a \in H^0(\bbP^1_A, \cO_{\bbP^1_A}(1))$ is non-vanishing on $S_f$). In particular, $h(s, -r) \in A^\times$ and we're done in this case. 

    We now suppose that $m > 0$. In this case, $S_f$ contains $\bbP^1_k$ as a closed subscheme and the sheaf $i^\ast \cO_a$ is non-zero. Since we assume that $\cM_\eta$ is free for each generic point $\eta$ of $S_f$, there exists an open subscheme $U$ of $S_f$, supported over $\Spec k$, and over which $\cM$ is free of rank 1. Let us choose $a$ so that $g(s, -r) \in A^\times$ and such that the support of $i^\ast \cO_a$ is contained in $U$. In this case, we again form the exact sequence (\ref{eqn_short_exact_sequence}), to obtain 
    \[ 0 \to M \to \Hom_A(M, A) \to \Hom_{\cO_{S_f}}(\cM, i^\ast \cO_a) \to 0.  \]
    To justify the exactness on the right, it is enough to show that $\Ext^1_{\cO_{S_f}}(\cM, \cO_{S_f}(n-3)) = 0$. The sheaf $\mathcal{E}\mathit{xt}^1_{\cO_{S_f}}(\cM, \cO_{S_f}(n-3))$ vanishes, by \cite[Proposition 1.6]{Har92} (note that $\cM$ is reflexive because there is an isomorphism $\cM \cong \ShHom_{\cO_{S_f}}(\cM, \cO_{S_f}(n-3))$, so we can apply \cite[Corollary 1.8]{Har92}).
    Therefore
    \[\Ext^1_{\cO_{S_f}}(\cM, \cO_{S_f}(n-3)) \cong H^1(S_f, \ShHom_{\cO_{S_f}}(\cM, \cO_{S_f}(n-3))) \cong H^1(S_f, \cM), \]
    and this group vanishes, by assumption. 
    
    By construction, $\Hom_{\cO_{S_f}}(\cM, i^\ast \cO_a)$ is isomorphic as $A$-module to the stalk of $i^\ast \cO_a$ at the point in $\bbP^1_k$ where $\overline{a}$ vanishes; in other words, to $A / (\pi^m)$. 
    By Lemma \ref{lemma_identificationconnectingmaps}, the morphism $M \to \Hom_A(M, A)$ in the above exact sequence is the one arising from $\mathrm{ev}_a \circ \varphi$, so we have shown that $\det(A_a) \in \pi^m A^\times$. Since $f(s, -r) \in \pi^m A^\times$, by construction, this completes the proof. 
\end{proof}

\section{The action of $\SL_n$ on pairs of symmetric matrices}\label{sec_the_representation}

\subsection{Basic definitions}\label{subsec_basicdefs}

Let $n\geq 2$ be an integer and let $W$ be the free $\Z$-module of rank $n$ with basis $e_1,\dots,e_n$.
Let $\Sym_2(W^{\vee})$ be the free $\Z$-module of symmetric bilinear forms $( -,-)\colon  W\times W\rightarrow \Z$.
Given such a form $(-,-)$, its Gram matrix with respect to the basis $\{e_i\}$ is the matrix $A = ((e_i,e_j))_{1\leq i,j\leq n}$. 
We use this to view $\Sym_2(W^{\vee})$ as the $\Z$-module of symmetric matrices $A\in \Sym_2(\Z^n)\subset\Mat_n(\Z)$.
Conversely, given a symmetric matrix $A\in \Sym_2(\Z^n)$, write $(-,-)_A$ for the corresponding symmetric bilinear form on $W$.

The group $\GL(W)$ and its subgroup $\SL(W)$ act on $W$ and hence by functoriality on $\Sym_2(W^{\vee}) = \Sym_2(\Z^n)$.
This action has the property that $(v,w)_{g\cdot  A} = (g^{-1} v,g^{-1}w)_A$ for all $v,w\in V$, $A\in \Sym_2(W^{\vee})$ and $g\in \GL(W)$. In terms of matrices, we have the formula $g\cdot A = g^{-t} A g^{-1}$.
The group $\SL(W)$ acts on $V$ via $g\cdot (A,B) = (g^{-t} Ag^{-1}, g^{-t} B g^{-1})$. 
The same formulae define actions of the group $\SL_n(R)$ on the $R$-modules $\Sym_2((W\otimes_{\Z} R)^{\vee}) = \Sym_2(R)$ and $V(R) = \Sym_2((W\otimes_{\Z} R)^{\vee}) \oplus \Sym_2((W\otimes_{\Z} R)^{\vee})$, for any ring $R$.
(Our formula for the $\SL(W)$-action on $V$ differs from \cite{bhargava2015mosthyperellipticarepointless, BGW17}, where they consider the action defined by $g\cdot (A,B) = (gAg^t, gBg^t)$, which coincides with our action after replacing $g$ by its inverse transpose. 
Since the orbits over any ring are the same for both actions, this difference is harmless.)

If $R$ is a ring and $A\in \Sym_2(R^n)$, its discriminant is by definition $\disc(A) = (-1)^{n(n-1)/2}\det(A)$.
We define the invariant binary form of $(A,B)\in V(R)$ to be
\begin{align*}
    f_{A,B}(x,y) = \disc(Ax-By)  = f_0 x^n + f_1 x^{n-1}y + \cdots + f_ny^n \in R[x,y].
\end{align*}
The coefficients of the invariant binary form are invariant polynomials of degree $n$ in the coefficients of $(A,B)$ for the $\SL_n(R)$-action, and freely generate the full ring of invariants when $R = \Z$.
If $f(x,y)\in R[x,y]$ is a binary form of degree $n$, write $V_f(R)$ for the $\SL_n(R)$-stable subset of $(A,B) \in V(R)$ satisfying $f_{A,B} = f$.

The discriminant $\Delta(f)$ of a binary form $f = f_0x^n + \cdots +f_ny^n$ is a homogeneous polynomial of degree $2n-2$ in $f_0, \dots,f_n$.
It is uniquely characterized by the property that if $R=K$ is an algebraically closed field and $f = \prod_{i=1}^n(\alpha_ix-\beta_i y)$ for some $\alpha_i, \beta_i \in K$, then 
\begin{align*}
    \Delta(f) = \prod_{i< j } (\alpha_i\beta_j - \alpha_j \beta_i)^2.
\end{align*}
If $(A,B)\in V(R)$ we set $\Delta(A,B) = \Delta(f_{A,B})$.

\subsection{Description of rational orbits}\label{subsec_rational_orbits}

Let $K$ be a field of characteristic zero and $f(x,y) = f_0 x^n + \cdots + f_n y^n \in K[x,y]$ a binary form of nonzero discriminant. Let $S_f \subset \bbP^1_K$ be the zero locus of $f$, $L_f = H^0(S_f, \cO_{S_f})$, an \'etale $K$-algebra of rank $n$, and define $I_f = H^0(S_f, \cO_{S_f}(n-3))$, $J_f = H^0(S_f, \cO_{S_f}(n-2))$, as in \S \ref{sec_binary_forms}. Then $I_f, J_f$ are free $L_f$-modules of rank 1, although they do not have distinguished generators. We write $I_f^\times \subset I_f$ for the set of elements which generate $I_f$ as an $L_f$-module; then $I_f^\times$ receives an action of $L_f^\times$. We define $J_f^\times$ similarly. 
\begin{definition}
    We call an $f$-module a tuple $(M, \varphi, e)$, where $M$ is a free $L_f$-module of rank 1, $\varphi : M \otimes_{L_f} M \to I_f$ is an isomorphism of $L_f$-modules, and $e \in \wedge^n_K M$ is a non-zero element, satisfying the following property: let $\mathcal{B} = \{ b_1, \dots, b_n \}$ be any $K$-basis of $L_f$ such that $b_1 \wedge \dots \wedge b_n = e$, and let $A_x, A_y$ be the matrices of the symmetric $K$-bilinear forms $\mathrm{ev}_x \circ \varphi$, $\mathrm{ev}_y \circ \varphi$ on $M$. Then $\disc(x A_y - y A_x) = f(x, y)$. 
\end{definition}
Henceforth we will write $\disc_e$ to denote discriminant with respect to some basis $\mathcal{B} = \{ b_1, \dots, b_n \}$ such that $b_1 \wedge \dots \wedge b_n = e$. Thus we can express the condition defining an $f$-module in the form $\disc_e( x \mathrm{ev}_y \circ \varphi - y \mathrm{ev}_x \circ \varphi) = f(x, y)$.

An isomorphism $(M, \varphi, e) \to (M', \varphi', e')$ of $f$-modules is an $L_f$-linear isomorphism $M \to M'$ intertwining the other data. We call $(A_y, A_x)$ the pair associated to the $f$-module $M$ and basis $\mathcal{B}$. 
\begin{proposition}\label{prop_bijectionorbitsfmodules}
    The map $(M, \varphi, e) \mapsto (A_y, A_x)$ determines a bijection between the following two sets:
    \begin{enumerate}
        \item The set of isomorphism classes of $f$-modules.
        \item The set of orbits $\SL_n(K) \backslash V_f(K)$.
    \end{enumerate}
\end{proposition}
\begin{proof}
    It is clear that the map $(1) \to (2)$ is well-defined. To show that it is bijective, we construct an inverse. Consider a pair $(A, B) \in V_f(K)$. We construct a structure $(M, \varphi, e)$ of $f$-module on $M = K^n$, together with basis $\mathcal{B} = \{ b_1, \dots, b_n \}$ with $b_1 \wedge \dots \wedge b_n = e$, such that $(A, B)$ is the pair associated to $(M, \varphi, e)$ and choice of basis $\mathcal{B}$. 

    Let $(s, -r) \in K^2$ be such that $s f(s, -r) \neq 0$, and let $a(x, y) = r x + s y$.  We set $A_a = s A + r B$. Then $\det A_a \neq 0$, and we give $M$ the unique $L_f$-module structure for which the element $x / a \in L_f$ (more precisely, the image of $x / a \in H^0(D^+(a), \cO_{\bbP^1_K})$ in $H^0(S_f, \cO_{S_f}) = L_f$) acts via the matrix $A_a^{-1} B$. We can then check the following:
    \begin{itemize}
     \item The element $y / a$ acts on $M$ via the matrix $A_a^{-1} A$. (Use the relation $r (x / a) + s (y / a) = 1$ in $L_f$.)
        \item The induced $L_f$-module structure on $M$ is independent of the choice of $a$. (If $b$ is another choice of linear form, use the relation $x / b = (b / a)^{-1}(x / a)$ in $L_f$ and check that the two possible actions of $x / b$ agree.)
    \end{itemize}
    We next define the pairing $\varphi$. Equivalently, we must define an isomorphism $M \cong \Hom_{L_f}(M, I_f)$ of $L_f$-modules. By duality (Proposition \ref{prop_application_of_global_duality}), there is a canonical isomorphism $\Hom_{L_f}(M, J_f) \cong \Hom_K(M, K)$, given by composition with $\zeta$. Multiplication by $a$ defines an isomorphism $\times a : I_f \to J_f$. The pairing given by the symmetric matrix $A_a$ gives an $L_f$-linear map $\psi_a : M \to \Hom_K(M, K) \cong \Hom_{L_f}(M, J_f)$. We define $\varphi = (\times a)^{-1} \circ \psi_a$. We can then check the following:
    \begin{itemize}
        \item The map $\varphi$ is an isomorphism of $L_f$-modules, which is independent of the choice of $a$. (If $b$ is another choice of linear form, we must check that $(b / a) \psi_a = \psi_b$. This is equivalent to the assertion that $A_a (b / a) = A_b$ as $K$-linear forms $M \times M \to K$, which follows from the formula for the action of $b / a$ on $M = K^n$.)
        \item The matrices of $A_y = \mathrm{ev}_y \circ \varphi$ and $A_x = \mathrm{ev}_x \circ \varphi$ with respect to the standard basis $\mathcal{B} = \{ b_1, \dots, b_n \}$ of $K^n$ are equal to $A$ and $B$, respectively. (This is essentially the same as the previous point, replacing $b$ by $y$ and using Lemma \ref{lemma_identificationconnectingmaps}.)
    \end{itemize}
    To complete our construction, we can therefore take $e = b_1 \wedge \dots \wedge b_n$. It is clear that this gives a well-defined map $(2) \to (1)$ (i.e.\ well-defined at the level of equivalence classes), which is inverse to the map $(1) \to (2)$. This completes the proof. 
\end{proof}
If $(M, \varphi, e)$ is an $f$-module, and $u \in M$ is an $L_f$-module generator, then $\varphi(u \otimes u) \in I_f^\times$. Multiplication by $u^{-1}$ determines an isomorphism $M \to L_f$ of $L_f$-modules, so we get an element $u^{-1}_\ast(e) \in \wedge^n_K L_f$. This gives a well-defined map from the set of isomorphism classes of $f$-modules to the set of equivalence classes in $I_f^\times \times (\wedge^n_K L_f - \{ 0 \})$, where we define pairs $(\alpha, z)$ and $(\alpha', z')$ to be equivalent if there exists $\beta \in L_f^\times$ such that $(\alpha', z') = (\beta^2 \alpha, \mathbf{N}_{L_f / K}(\beta)^{-1} z)$. 
\begin{proposition}\label{prop_base_point_free_description_of_orbits} The map just defined has the following properties:
    \begin{enumerate}
        \item It is injective.
        \item Let $(\alpha, z) \in I_f^\times \times (\wedge^n_K L_f - \{ 0 \})$. Then the equivalence class of $(\alpha, z)$ lies in the image if and only if $\disc_z( x (\mathrm{ev}_y \circ \varphi_\alpha) - y (\mathrm{ev}_x \circ \varphi_\alpha) ) = f(x, y)$, where $\varphi_\alpha : L_f \otimes_{L_f} L_f \to I_f$ is defined by $\varphi_\alpha(a \otimes b) = a b \alpha$.
    \end{enumerate}
\end{proposition}
\begin{proof}
    Suppose that $f$-modules $(M, \varphi, e)$ and $(M', \varphi', e')$ have the same image. It follows that we can choose basis elements $u, u'$ for $M, M'$ respectively such that $\varphi(u \otimes u) = \varphi'(u' \otimes u')$ and $u^{-1}_\ast(e) = (u')^{-1}_\ast(e')$. Then the isomorphism $M \to M'$ of $L_f$-modules sending $u$ to $u'$ intertwines the other structures, so is an isomorphism of $f$-modules.

    The characterisation of the image is essentially the definition of an $f$-module. 
\end{proof}
We remark that if $M$ is a free $L_f$-module of rank 1, $\varphi : M \otimes_{L_f} M \to I_f$ is an isomorphism, and $e \in \wedge^n_K M$ is non-zero, then $\disc_e( x (\mathrm{ev}_y \circ \varphi) - y (\mathrm{ev}_x \circ \varphi) )$ is a $K^\times$-multiple of $f(x, y)$. To check equality, it is therefore enough to check that 
\[ \disc_e( s (\mathrm{ev}_y \circ \varphi) + r (\mathrm{ev}_x \circ \varphi) ) = \disc_e( \mathrm{ev}_{r x + s y} \circ \varphi) = f(s, -r) \]
for some $(r, s) \in K^2$ such that $f(s, -r) \neq 0$. 
\begin{corollary}\label{cor_unique_geometric_orbit}
    Suppose that $K$ is algebraically closed. Then the set $\SL_n(K) \backslash V_f(K)$ has exactly one element.
\end{corollary}
\begin{proof}
    Since $L_f$ is an \'etale $K$-algebra and $K$ algebraically closed, $L_f\simeq K \times \cdots \times K$, every element of $L_f$ is a square and any two elements of $I_f^{\times}$ differ by a square in $L_f^{\times}$.
    Moreover, for every $\alpha \in I_f^{\times}$ there exists $e\in (\wedge^n_K L_f - \{ 0 \})$ such that $(L_f, \varphi_{\alpha}, e)$ is an $f$-module, in the notation of Proposition \ref{prop_base_point_free_description_of_orbits}.
    If $e'\in (\wedge^n_K L_f - \{ 0 \})$ is such that $(L_f, \varphi_{\alpha}, e')$ is an $f$-module, then $e' = \pm e$.
    Therefore every pair $(\alpha',z') \in I_f^\times \times (\wedge^n_K L_f - \{ 0 \})$ is equivalent to the fixed pair $(\alpha, e)$.
\end{proof}

A common special case is when $f_0 \neq 0$ (so in particular $f(1, 0) \neq 0$).  Then various objects become more explicit, because $S_f$ is contained in the distinguished affine open $D^+(y) \subset \bbP^1_K$. Taking $X = x / y$, we find that $L_f = K[X] / (f(X, 1))$. Thus $L_f$ has the power basis $1, X, \dots, X^{n-1}$, while $I_f, J_f$ have $L_f$-basis the elements $y^{n-3}$, $y^{n-2}$, respectively. 
\begin{lemma}\label{lem_computation_of_zeta}
    With notation as in the previous paragraph, the map $\zeta : J_f \to K$ is equal to $f_0^{-1} (x^{n-1} y^{-1})^\ast$, where $(y^{n-2})^\ast, (x y^{n-3})^\ast, \dots, (x^{n-1} y^{-1})^\ast$ is the $K$-basis of $\Hom_K(J_f, K)$ dual to the $K$-basis $y^{n-2}$, $x y^{n-3}, \dots, x^{n-1} y^{-1}$ of $J_f$.
\end{lemma}
\begin{proof}
    We compute using the Cech cohomology of the affine covering $U_0 = D^+(y)$, $U_1 = D^+(x)$ of $\P^1_K$. 
    Our hypothesis $f_0 \neq 0$ is equivalent to the condition $S_f \subset D(y)$. Thus the map 
    \[ H^0(S_f, \cO_{S_f}(n-2)) \to H^0(U_0, \cO_{S_f}(n-2)) = K[X] \cdot y^{n-2} / K[X] \cdot y^{-2} f(x, y) = K[X] \cdot y^{n-2} / f(X, 1 )K[X] \cdot y^{n-2} \]
    is an isomorphism. This shows that the elements $y^{n-2}, x y^{n-3}, \dots, x^{n-1} y^{-1} \in K[X] \cdot y^{n-2}$ project to a $K$-basis of $J_f$. Since $(xy)^{-1} f(x, y) \in H^0(U_0 \cap U_1, \cO_{\P^1_K}(n-2))$ maps to $0$ in $H^0(U_0 \cap U_1, \cO_{S_f}(n-2))$, the Cech cocycle in 
    \[ H^0( U_0, \cO_{S_f}(n-2)) \oplus H^0 (U_1, \cO_{S_f}(n-2))  \]
    corresponding to $x^{n-1} y^{-1}$ is $(x^{n-1} y^{-1}, (  x^{n-1} y^{-1} - f_0^{-1}(xy)^{-1} f(x, y)))$.

    Recall that $\zeta$ is defined as the composite of the boundary map $J_f = H^0(S_f, \cO_{S_f}(n-2)) \to H^1(\P^1_K, \cO_{S_f}(-2))$ associated to the short exact sequence
    \[ 0 \to \cO_{\P^1_K}(-2) \overset{\times f}{\to} \cO_{\P^1_K}(n-2) \to \cO_{S_f}(n-2) \to 0 \]
    with the isomorphism $H^1(\P^1_K, \cO_{\P^1_K}(-2)) \cong K$ afforded by Serre duality. In particular, it vanishes on the codimension 1 subspace $H^0(\P^1_K, \cO_{\P^1_K}(n-2))$ of $J_f$. To prove the lemma, we need to show that $\zeta$ takes the value $f_0^{-1}$ on $x^{n-1} y^{-1}$. 
    
    The image of $x^{n-1} y^{-1}$ under the boundary homomorphism is represented by the cocycle 
    \[ (x^{n-1} y^{-1} - (  x^{n-1} y^{-1} - f_0^{-1}(xy)^{-1} f(x, y)))/f(x, y) = f_0^{-1} (xy)^{-1} \in H^0(U_0 \cap U_1, \cO_{\P^1_K}(-2)). \]
    By definition, the map $\zeta$ takes $(xy)^{-1}$ to $1$, so it follows that $\zeta(x^{n-1} y^{-1}) = f_0^{-1}$, as claimed. 
\end{proof} 
In this way, we recover the description of the set $\SL_n(K) \backslash V_f(K)$ given in the special case $f_0 \neq 0$ in \cite[Theorem 7]{BGW15}:
\begin{corollary}\label{cor_base_point_description_of_rational_orbits}
    Suppose that $f_0 \neq 0$. Then the following three sets are in bijection:
    \begin{enumerate}
         \item The set of orbits $\SL_n(K) \backslash V_f(K)$.
        \item The set of isomorphism classes of $f$-modules $(M, 
        \varphi, e)$.
        \item The set of equivalence classes of pairs $(\alpha, z) \in L_f^\times \times K^\times$ satisfying $z^{2} \mathbf{N}_{L_f / K}(\alpha) = f_0^{n+1}$, where we say $(\alpha, z)$, $(\alpha', z')$ are equivalent if there exists $\beta \in L_f^\times$ such that $\alpha' = \beta^2 \alpha$ and $z' = \mathbf{N}_{L_f / K}(\beta)^{-1} z$. 
    \end{enumerate}
\end{corollary}
\begin{proof}
    Let $e = 1 \wedge X \wedge \dots \wedge X^{n-1} \in \wedge^n_K L_f$. We define a bijection $L_f^\times \times K^\times \to I_f^\times \times (\wedge^n_K L_f - \{ 0 \})$ by the formula 
    \[ (\alpha, z) \mapsto (\alpha y^{n-3}, z e). \]
    This descends to a bijection between the corresponding equivalence classes. If $\alpha \in L_f^\times$, let $\varphi_\alpha : L_f \otimes_{L_f} L_f \to I_f$ be the map $u \otimes v \mapsto \alpha u v y^{n-3}$. 
    
    What we need to check is that the condition $\disc_{ze}( \mathrm{ev}_y \circ \varphi_\alpha) = f_0$ arising from Proposition \ref{prop_base_point_free_description_of_orbits} is equivalent to the condition $z^{2} \mathbf{N}_{L_f / K}(\alpha) = f_0^{n+1}$ appearing in the statement of Corollary \ref{cor_base_point_description_of_rational_orbits}. This is a computation: 
    \[ \disc_{ze}( \mathrm{ev}_y \circ \varphi_\alpha) = z^2 \disc_e ( \zeta(\alpha u v y^{n-2}) ) = z^2 \mathbf{N}_{L_f / K}(\alpha) \disc_e (\zeta (u v y^{n-2} ) ), \]
    where we write e.g. $\zeta (u v y^{n-2})$ for the $K$-bilinear form $L_f \times L_f \to K$, $(u, v) \mapsto \zeta (u v y^{n-2} )$. Using Lemma \ref{lem_computation_of_zeta}, we see that $\zeta (u v y^{n-2}) = f_0^{-1} \tau(uv)$, where $\tau : L_f \to K$ is the element of the dual basis of the $K$-basis $1, X, \dots, X^{n-1}$ of $L_f$ corresponding to $X^{n-1}$. We have $\disc_e( \tau(uv)) = 1$, hence $\disc \zeta (u v y^{n-2}) = f_0^{-n}$, and finally
    \[ \disc_{ze}( \mathrm{ev}_y \circ \varphi_\alpha) = z^2 \mathbf{N}_{L_f / K}(\alpha) f_0^{-n}. \]
    This is what we needed. 
\end{proof}

\subsection{Integral orbits}

We will use the point of view developed in \S\ref{subsec_rational_orbits} to construct elements of $\SL_n(K) \backslash V_f(K)$ from $K$-rational divisors on hyperelliptic curves in \S \ref{sec_integral_orbit_representatives} below. We conclude with a lemma that will be useful in extending this construction to produce integral orbits. Since we now want to consider integral structures, let us reset our notation. Let  $A$ be a PID with fraction field $K$ of characteristic 0, and let $f(x, y) = f_0 x^n + \dots + f_n y^n \in A[x, y]$ be a homogeneous polynomial of degree $n$ and non-zero discriminant $\Delta(f) \in A$. Let $S_f \subset \bbP^1_A$ denote the zero locus of $f$, and let $S_{f, K}$ denote its generic fibre. Similarly, let $I_f = H^0(S_f, \cO_{S_f}(n-3))$, $I_{f, K} = I_f \otimes_A K$ its generic fibre. 
\begin{lemma}\label{lem_application_of_integral_lattice}
    Let $(M_K, \varphi_K, e)$ be an $f$-module, and suppose given an $A$-lattice $M \leq M_K$ satisfying the following conditions:
    \begin{enumerate}
        \item The restriction $\varphi$ of the map $\varphi_K : M_K \times M_K \to I_{f, K}$ to $M \times M$ takes values in $I_f$.
        \item There exists an $A$-basis $\mathcal{B}$ of $M$ such that $\disc_\mathcal{B}( x (\mathrm{ev}_y \circ \varphi) - y (\mathrm{ev}_x \circ \varphi) ) = u f(x, y)$ for some $u \in A^\times$.  
    \end{enumerate}
    Then:
    \begin{enumerate}
        \item There exists an $A$-basis $\mathcal{B}' = \{ b_1', \dots, b_n' \}$ of $M$ such that $e = b_1' \wedge \dots \wedge b_n'$.
        \item If $A_y,A_x$ are the matrices of $\mathrm{ev}_y \circ \varphi$, $\mathrm{ev}_x \circ \varphi$ with respect to the basis $\mathcal{B}'$, then $(A_y,A_x)$ is a representative in $V_f(A)$ for the $\SL_n(K)$-orbit corresponding to the $f$-module $(M_K, \varphi_K, e)$ under Proposition \ref{prop_bijectionorbitsfmodules}.
    \end{enumerate}
\end{lemma}
\begin{proof}
    Let $\mathcal{C} = \{ c_1, \dots, c_n \}$ be a $K$-basis of $M_K$ such that $c_1 \wedge \dots \wedge c_n = e$, so that 
    \[ \disc_\mathcal{C}( x (\mathrm{ev}_y \circ \varphi) - y (\mathrm{ev}_x \circ \varphi) ) = f(x, y). \]
    If $g \in \GL_n(K)$ is the change of basis matrix for $\mathcal{B}$, $\mathcal{C}$, then $\disc_\mathcal{B} = \det(g)^2 \disc_\mathcal{C}$, hence $u = \det(g)^2$, and so $u = v^2$ for some $v \in A^\times$. It follows that we can choose e.g. $b_1' = v^{-1} b_1$, $b_i' = b_i$ if $i = 2, \dots, n$. 
\end{proof}

\subsection{Reduction theory}\label{subsec: reduction covariant def}

We define a reduction covariant for the representation $V(\Z)$ of $\SL_n(\Z)$, a key technical tool in our reduction theory.
To this end, we start with a slightly more general set-up.

Let $M$ be an $n$-dimensional $\R$-vector space equipped with two symmetric bilinear forms $(-,-)_1, (-,-)_2\colon M\times M\rightarrow \mathbb{R}$.
Suppose that there exists a basis of $M$ such that the Gram matrices $A,B$ of $(-,-)_1, (-,-)_2$ have the property that $\det(A x - B y)\in \R[x,y]$ has nonzero discriminant.
We explain how to associate to this data a canonical positive definite inner product $H$ on $M$.
\begin{lemma}\label{lemma_existencediagonalbasis}
    There exists a $\mathbb{C}$-basis $b_1, \dots, b_n$ of $M \otimes_{\R} \mathbb{C}$ such that the Gram matrices of $(-,-)_1, (-,-)_2$ are both diagonal. 
    Moreover, such a basis is unique up to reordering and rescaling.
\end{lemma}
\begin{proof}
    It suffices to prove these claims when $M = \mathbb{R}^n$, in which case we may view the pair $(-,-)_1, (-,-)_2$ as an element $(A,B)\in V(\mathbb{R})$ with invariant binary form $f(x,y) \in \mathbb{R}[x,y]$ of nonzero discriminant. 
    By Proposition \ref{prop_bijectionorbitsfmodules} and its proof, $M$ is part of an $f$-module $(M, \varphi, e)$ with $\mathrm{ev}_y \circ \varphi = A$ and $\mathrm{ev}_x\circ \varphi = B$.
    Let $v\in M$ be an $L_f$-module generator. 
    Let $e_1, \dots, e_n$ be the idempotents corresponding to a decomposition $L_f \otimes_{\R}\mathbb{C} \simeq \mathbb{C} \times \cdots \times \mathbb{C}$.
    Then $e_1 \cdot v, \dots, e_n\cdot  v$ is a $\mathbb{C}$-basis of $M \otimes_{\R} \mathbb{C}$ and $\varphi(e_i\cdot v, e_j\cdot  v)=e_ie_j \varphi(v,v) = 0$ if $i\neq j$, so the forms $A,B$ are both diagonal in this basis, proving existence.
    To prove uniqueness (up to reordering and rescaling), let $(s,-r) \in \R^2$ with $sf(s,-r)\neq 0$, let $a(x,y)=rx + sy$ and let $A_a = s A + rB$.
    Then $\det(A_a)\neq 0$, and any basis in which $A,B$ are diagonal has the property that $A_a^{-1} B$ is also diagonal. 
    Since $\disc(f) \neq 0$, $A_{a}^{-1} B$ is regular semisimple, hence such a basis is unique up to reordering and rescaling.
\end{proof}

\begin{construction}\label{con_reduction_covariant}
Let $b_1, \dots, b_n$ be a $\mathbb{C}$-basis of $M \otimes_{\R} \mathbb{C}$ for which the Gram matrices of $(-,-)_1, (-,-)_2$ are both diagonal. 
The assignment
\begin{align*}
    (b_i,b_j)_H = 
    \begin{cases}
        \max(|(b_i,b_i)_1, |(b_i,b_i)|_2) & \text{ if } i=j, \\
        0 & \text{ if } i\neq j,
    \end{cases}
\end{align*}
extends to a Hermitian form on $M \otimes_{\R} \mathbb{C}$.
Let $H$ be the restriction of this Hermitian form to $M$.
\end{construction}

\begin{proposition}\label{prop_basicpropsreductioncovariant}
    In the above construction, $H$ is independent of the choice of $b_1,\dots,b_n$, real valued and positive definite. 
    Moreover, if $(M', (-,-)_1', (-,-)_2')$ is another such triple and $\psi \colon (M,(-,-)_1, (-,-)_2)\rightarrow (M', (-,-)_1', (-,-)_2')$ is an isomorphism of vector spaces intertwining the respective forms, then $\psi$ intertwines $H$ and $H'$.
\end{proposition}
\begin{proof}
    By Lemma \ref{lemma_existencediagonalbasis}, the basis $b_1,\dots,b_n$ is unique up to reordering and rescaling, and these operations do not affect $H$.
    Choosing the basis so that $b_1,\dots,b_k\in M$ and $b_{k+1}, \dots,b_n$ come in complex conjugate pairs shows that $H$ is real valued.
    To show that $H$ is positive definite, we must show that $\max(|(b_i,b_i)_1|, |(b_i,b_i)_2|)>0$ for each $i=1,\dots,n$.
    This is true, since the nonzero discriminant condition implies (\cite[Lemma 6.4(b)]{BSW-squarefreeII}) that either $(b_i,b_i)_1$ or $(b_i,b_i)_2$ is nonzero for each $i$.
    The claim about compatibility with isomorphisms follows from the fact that the definition of $H$ is independent of any additional choices.
\end{proof}

Applying this construction to our framed vector space $W_{\R} =\R^n$ fixed in \S\ref{subsec_basicdefs}, we obtain an inner product $H_{A,B}$ on $W_{\R}$ associated to any pair $(A,B)\in V(\R)$ of nonzero discriminant.
Let $X$ denote the set of equivalence classes of inner products on $W_{\R}$, where two inner products $H$, $H'$ are defined to be equivalent when one is an $\R_{>0}$-multiple of the other.
Write $[H] \in X$ for the image of an inner product $H$ on $W_{\R}$.
The assignment $(A,B)\rightarrow [H_{A,B}]$ defines a map 
\begin{align}
    \mathcal{R}\colon V(\R)^{\Delta\neq 0} \rightarrow X,
\end{align}
which we call the reduction covariant.
The group $\SL_n(\R)$ acts on $X$ via the formula $(v,w)_{g\cdot H} = (g^{-1}v,g^{-1}w)_{H}$ or in terms of matrices, via the formula $g\cdot H = g^{-t} H g^{-1}$.
Let $H_0$ be the standard inner product on $W_{\R}$ with Gram matrix the identity matrix.
The map $g\mapsto g\cdot H_0$ induces an $\SL_n(\R)$-equivariant bijection $\SL_n(\R)/\SO_n(\R) \simeq X$.
The reduction covariant $\mathcal{R}$ is $\SL_n(\R)$-equivariant by the last sentence of Proposition \ref{prop_basicpropsreductioncovariant}, hence by passage to the quotient we obtain a map 
\begin{align}\label{equation: reduction covariant integral orbits}
    \mathcal{R}\colon \SL_n(\Z)\backslash V(\Z)^{\Delta\neq 0} \rightarrow \SL_n(\Z) \backslash X.
\end{align}

We record a computation of $\det H_{A,B}$ with respect to the standard basis of $W_{\R}$.
Let $(A,B) \in V(\R)$ with $\Delta(A,B) \neq 0$ and write
\begin{align*}
    f_{A,B} = c \prod_{i=1}^r (x-\omega_i y) \prod_{j=1}^k (\eta_j x- y),
\end{align*}
where $c, \omega_i, \eta_j \in \mathbb{C}$ and $|\omega_i|, |\eta_j|\leq 1$ for all $i,j$.
\begin{lemma}\label{lemma: determinant of reduction covariant}
    In the above notation, $\det(H_{A,B}) = |c|$.
\end{lemma}
\begin{proof}
    Let $c_0\in \mathbb{C}$ be an element such that $c_0^n = (-1)^{n(n-1)/2}c$ and consider the pair of diagonal matrices $(A_0,B_0)\in V(\mathbb{C})$, where $A_0 = (c_0,\dots, c_0, c_0\eta_1,\dots,c_0\eta_k)$ and $B_0 = (c_0 \omega_1, \dots, c_0\omega_r, c_0,\dots, c_0)$.
    A computation shows that $f_{A_0,B_0} = f_{A,B}$.
    By Corollary \ref{cor_unique_geometric_orbit}, there exists a $g\in \SL_n(\mathbb{C})$ such that $g\cdot (A,B) = (A_0,B_0)$. 
    By definition of the reduction covariant, $g^{-t}H_{A,B}(\bar{g})^{-1}$ is a diagonal matrix, with $i$th diagonal entry equal to the maximum of the absolute values of the $i$th diagonal entries of $A_0$ and $B_0$. Therefore $g^{-t}H_{A,B}(\bar{g})^{-1}$ is the diagonal matrix $(|c_0|,\dots,|c_0|)$, so $\det(H_{A,B}) = \det(g^{-t}H_{A,B}(\bar{g})^{-1}) = |c_0|^n = |c|$.
\end{proof}
    We remark that the reduction covariant given above is not unique. For example, for either choice of $k = 1, 2$ we could carry out the analogue of Construction \ref{con_reduction_covariant} with the formula
    \begin{align*}
    (b_i,b_j)_{H_k} = 
    \begin{cases}
        |(b_i,b_i)_k| & \text{ if } i=j, \\
        0 & \text{ if } i\neq j,
    \end{cases}
\end{align*}
    to obtain a different $\SL_n(\R)$-equivariant reduction covariant $\mathcal{R}_k : V(\R)^{\Delta \neq 0} \to X$. Computation suggests that carrying out reduction theory using $\mathcal{R}_1$ will make the coefficients of $A$ as small as possible (perhaps at the cost of making the coefficients of $B$ large), and vice versa for $\mathcal{R}_2$, while performing reduction theory with $\mathcal{R}$ will tend to reduce the coefficients of both $A$ and $B$. Our choice of $\mathcal{R}$ is made here so that we can prove Lemma \ref{lem_reduction_covariant_constant_on_section}.

\section{Construction of rational and integral orbits}\label{sec_integral_orbit_representatives}

We now connect the representation of $\SL_n$ on $V$ to divisors on hyperelliptic curves.
Let $A$ be a PID of fraction field $K$ of characteristic 0, let $g \geq 1$, and let $n = 2g + 2$. Let 
\[ f(x, y) = f_0 x^n + f_1 x^{n-1} y + \dots + f_n y^n \in A[x, y] \]
be a homogeneous form of degree $n$ such that the discriminant $\Delta(f) \in A$ is nonzero. Let $\pi : X \to \mathbb{P}^1_A$ denote the double cover described by the equation $z^2 = f(x, y)$, and $\iota : X \to X$ the associated involution. Then $X_K$ is a smooth hyperelliptic curve. We write $j : S_f \hookrightarrow X$ for the closed subscheme defined by the equation $z = 0$. As suggested by the notation, $S_f$ may be identified with the subscheme of $\bbP^1_A$ defined by the vanishing of $f \in H^0(\bbP^1_A, \cO_{\bbP^1_A}(n))$, studied in \S \ref{sec_binary_forms}. We set $R_f = H^0(S_f, \cO_{S_f})$ and $L_f = R_f \otimes_A K$. 

Let $D_K \leq X_K$ be an effective divisor of degree $2g-1$, not intersecting $S_{f, K}$. We can write $D_{\overline{K}} = \sum_{i=1}^k (\alpha_i : 1 : \beta_i) + \sum_{i=k+1}^{2g-1} (1 : 0 : \beta_i)$ for some numbers $\beta_i \in \overline{K}$. We set $w = \prod_{i=1}^{2g-1} \beta_i \in K$.

Let $U(x, y) = u_0 x^{2g-1} + \dots + u_{2g-1} y^{2g-1} \in A[x, y]$ be a primitive form of degree $2g-1$ vanishing precisely at the points of $\pi(D_{\overline{K}})$ (with multiplicity). Then $U$ is determined up to multiplication by $A^\times$. In this section, we will prove the following theorem.
\begin{theorem}\label{thm_existence_of_integral_representatives}
    Suppose that $f_0 \neq 0$, let $X = x / y$, and let $\alpha = U(X, 1) \text{ mod }f(X, 1) \in L_f = R_f \otimes_A K$. Let $w = \prod_{i=1}^{2g-1} \beta_i$. Then:
    \begin{enumerate}
        \item The pair $(\alpha, u_k^{-(g+1)} f_0^{n-1} w^{-1})$ corresponds, under the bijection of Corollary \ref{cor_base_point_description_of_rational_orbits}, to an  orbit in $\SL_n(K) \backslash V_f(K)$.
        \item Furthermore, this orbit intersects $V_f(A)$ (i.e.\ has an integral representative).
    \end{enumerate}
\end{theorem}
After the proof, we will give a more precise version of Theorem \ref{thm_existence_of_integral_representatives} (see Theorem \ref{thm_tautological_section_is_saturated}), which is what we will use in our applications. The rational orbit given by Theorem \ref{thm_existence_of_integral_representatives} could depend on the choice of $U$, although its orbit under the group $(\SL_n / \mu_2)(K)$ (through which the action of $\SL_n(K)$ factors) is insensitive to this choice. 

We begin by proving the first part of Theorem \ref{thm_existence_of_integral_representatives}, which is a computation. Let $\omega_1, \dots, \omega_n \in \overline{K}$ be the roots of $f(X, 1)$. Then $U(X, 1) = u_k \prod_{i=1}^{k} (X - \alpha_i)$ and $f(X, 1) = f_0 \prod_{j=1}^n (X - \omega_j)$. Since $f_0 \neq 0$, we have $w \in K^\times$, and $w^2 = f_0^{n-3-k} \prod_{i=1}^k \beta_i^2$. 

We have 
\begin{multline*}  \mathbf{N}_{L_f / K}(\alpha) = \prod_{j=1}^n U(\omega_j, 1) = u_k^{n} \prod_{i=1}^{k} \prod_{j=1}^n (\omega_j - \alpha_i) \\
=  u_k^{n} \prod_{i=1}^{k} \prod_{j=1}^n (\alpha_i - \omega_j) = u_k^n f_0^{-k} \prod_{i=1}^{k} f(\alpha_i, 1) = u_k^{n} f_0^{-k} \prod_{i=1}^{k} \beta_i^2 = u_k^n w^2 f_0^{3-n}.
\end{multline*} 
It follows that $(u_k^{-(g+1)} w^{-1} f_0^{n-1})^2 \mathbf{N}_{L_f / K}(\alpha) = f_0^{n+1}$. Looking at the statement of Corollary \ref{cor_base_point_description_of_rational_orbits}, we see that we have indeed obtained an orbit in $\SL_n(K) \backslash V_f(K)$. In order to show that this orbit admits an integral representative, we first give a geometric construction for its corresponding $f$-module (in the sense of \S \ref{subsec_rational_orbits}), in a way similar to \cite{Tho14, lagathorne2024smallheightoddhyperelliptic}. Let $M_K = H^0(S_{f, K}, j_K^\ast \cO_{X_K}(D_K))$. Then $M_K$ is a free $L_f$-module of rank 1. There is an isomorphism of sheaves 
\begin{equation}\label{eqn_isomorphism_of_line_bundles_in_generic_fibre} \cO_{X_K}(D_K) \otimes_{\cO_{X_K}} \cO_{X_K}(\iota^\ast D_K) \cong \cO_{X_K}(2g-1) 
\end{equation}
given by $a \otimes b \mapsto a b U$. After pullback along $j_K^\ast : S_{f, K} \hookrightarrow X_K$, this gives (using the identity $\iota \circ j = j$) an isomorphism of sheaves
\[ j_K^\ast \cO_{X_K}(D_K) \otimes_{\cO_{S_{f, K}}} j_K^\ast \cO_{X_K}(D_K) \cong j_K^\ast \cO_{X_K}(2g-1), \]
hence an isomorphism
\[ M_K \cong \Hom_{L_f}(M_K, I_f \otimes_A K) \]
of $L_f$-modules, which is determined up to $A^\times$-multiple by $D_K$ (and determined completely by $D_K$ and the choice of $U(x, y)$). Let $\varphi_K : M_K \otimes_{L_f} M_K \to I_f \otimes_A K$ denote the corresponding isomorphism. The first part of Theorem \ref{thm_existence_of_integral_representatives} says that $(M_K, \varphi_K, e)$ is an $f$-module, where $e \in \wedge^n_K M_K$ is defined by the formula
\[ e = u_0^{-(g+1)} f_0^{g} w^{-1} \cdot ( \overline{1} \wedge X \overline{1} \wedge \dots \wedge X^{n-1} \overline{1}), \]
and $\overline{1} \in M_K$ is the image of the tautological section $1 \in H^0(X_K, \cO_{X_K}(D_K))$ under restriction. Indeed, our hypothesis that $D_K$ and $S_f$ do not intersect implies that $\overline{1}$ is an $L_f$-module generator for $M_K$.

We now proceed to construct an $A$-lattice $M \leq M_K$. We will rely heavily on the fact that $X$ is an integral local complete intersection of dimension 2. Continuing to take $D_K$ to be an effective divisor in $X_K$ of degree $2g-1$, let us take $D$ to be its Zariski closure in $X$, a closed subscheme of $X$ which is $A$-flat of degree $2g-1$. Let $Z \subset X$ denote the intersection of the subscheme cut out by $U \in H^0(X, \cO_X(2g-1))$ with the finitely many fibres of $X \to \Spec A$ above points where $\Delta(f)$ is not a unit. Then $Z$ has codimension 2; let $V = X - Z$. Then $D_V = D \cap V$ is a Cartier divisor in $V$, and (by \cite[Theorem 1.12]{Har92}) the restriction functor $\mathrm{Refl}(X) \to \mathrm{Refl}(V)$ between the respective categories of reflexive coherent sheaves is an equivalence. 
Consequently, there is a unique (up to unique isomorphism) reflexive coherent sheaf, that we call $\cO_X(D)$, on $X$, whose restriction to $V$ is isomorphic to $\cO_V(D_V)$. Moreover, since $D_V + \iota^\ast D_V = (U)_V$ as Cartier divisors, the isomorphism $\cO_{X_K}(D_K) \cong \ShHom_{\cO_{X_K}}(\cO_{X_K}(\iota^\ast D), \cO_{X_K}(2g-1))$ extends to an isomorphism 
\begin{equation}\label{eqn_pairing_of_sheaves_on_X} \cO_X(D) \cong \ShHom_{\cO_X}(\cO_X(\iota^\ast D), \cO_X(2g-1)) 
\end{equation}
of reflexive coherent sheaves of $\cO_X$-modules (here we use that the sheaf Hom is itself reflexive, by \cite[Corollary 1.18]{Har92}). 

Let $\mathcal{M} = j^\ast \cO_X(D)$. We claim that $\mathcal{M}$ satisfies the hypotheses of Proposition \ref{prop_pencil_of_quadrics_over_A}. Combining the result of this proposition with Lemma \ref{lem_application_of_integral_lattice} will imply the truth of Theorem \ref{thm_existence_of_integral_representatives}. These hypotheses be checked locally on $\Spec A$, so we now assume that $A$ is a DVR of uniformizer $\varpi$ and residue field $k = A / (\varpi)$. 
\begin{lemma}
    $\cM_\eta$ is free of rank 1 at each generic point $\eta \in S_f$.
\end{lemma}
\begin{proof}
    By construction, $\cM$ is locally free of rank 1 after restriction to $V = X - Z$. It suffices to check that the generic points of $S_f$ all lie in $V$. Since $V$ contains $X_K$, it certainly contains all generic points of $S_f$ lying above the generic point of $\Spec A$. $S_f$ has generic points lying above $\Spec k$ if and only if $f$ is not primitive, in which case there is a generic point corresponding to the closed subscheme $\bbP^1_k$ of $S_f$. The intersection $\bbP^1_k \cap Z$ is finite (being contained in the zero set of $\overline{U}(x, y) \in k[x, y]$), so we see that we're done in this case too. 
\end{proof}
\begin{lemma}\label{lem_pullback_of_pairing_to_Weierstrass_locus}
    The isomorphism of Equation (\ref{eqn_pairing_of_sheaves_on_X}) determines, by pullback, an isomorphism
    \[ \mathcal{M} \cong \ShHom_{\cO_{S_f}}(\mathcal{M}, \cO_{S_f}(2g-1))  \]
    of sheaves of $\cO_{S_f}$-modules.
\end{lemma}
\begin{proof}
    We need to show that the natural map
    \[ j^\ast \ShHom_{\cO_X}(\cO_X(\iota^\ast D), \cO_X(2g-1)) \to \ShHom_{\cO_{S_f}}( j^\ast \cO_X(\iota^\ast D), j^\ast \cO_X(2g-1)) \]
    is an isomorphism. We can work on stalks. Let $p \in S_f$, and let $u$ be a local equation for $S_f$ at $p$. Applying the functor $\Hom_{\cO_{X, p}}(\cO_X(\iota^\ast D)_p, -)$ to the short exact sequence
    \[ 0 \to \cO_{X, p} \overset{\times u}{\to} \cO_{X, p} \to \cO_{S_f, p} \to 0, \]
    we see that it is enough to know that $\Ext^1_{\cO_{X, p}}(\cO_X(\iota^\ast D)_p, \cO_{X, p}) = 0$. This will follow from \cite[Corollary 1.3]{Har92} if we can check that $\cO_X(\iota^\ast D)_p$ is Cohen--Macaulay. However, \cite[Theorem 1.9]{Har92} shows that since $\cO_X(\iota^\ast D)_p$ is reflexive, it is $S_2$, hence Cohen--Macaulay. 
\end{proof}
\begin{lemma}\label{lem_vanishing_cohomology}
    $H^1(S_f, \mathcal{M}) = 0$.
\end{lemma}
\begin{proof}
    If $f$ is primitive, then $S_f$ is affine, so there is nothing to prove. We therefore assume that $f(x, y) = \varpi^m g(x, y)$ for some $m \geq 1$, where $g(x, y) \in A[x, y]$ is primitive. 
    Taking in hand the short exact sequence
    \[ 0 \to \cO_X(D)(-(g+1)) \overset{\times z}{\to} \cO_X(D) \to \mathcal{M} \to 0, \]
    we see that it is enough to show that $H^1(X, \cO_X(D)) = 0$. 
    By cohomology and base change, it is even enough to show that $H^1(X_k, \mathcal{L}) = 0$, where $\mathcal{L} = \cO_X(D)|_{X_k}$.
    The curve $X_k$ is the genus $g$ `hyperelliptic ribbon' \cite[\S1]{BayerEisenbud-ribbons} given by the equation $z^2 = 0$, where $z \in \cO_{\bbP_k^1}(-(g+1))$. Writing $\pi : X_k \to \bbP^1_k$ for the double cover, $\pi_\ast \mathcal{L}$ is a torsion free coherent sheaf, generically of rank 2, therefore a locally free sheaf of rank 2. (We justify the assertion that $\pi_\ast \mathcal{L}$ is torsion-free. By \cite[\href{https://stacks.math.columbia.edu/tag/0AUV}{Lemma 0AUV}]{stacks-project}, it is equivalent to show that $(0)$ is the unique associated prime of each stalk of $\pi_\ast \mathcal{L}$. By \cite[\href{https://stacks.math.columbia.edu/tag/05DZ}{Lemma 05DZ}]{stacks-project}, it is equivalent to show that the unique associated prime of each stalk of $\mathcal{L}$ is the unique minimal prime. The sheaf $\mathcal{L}$ is $S_1$, since $\cO_X(D)$ is $S_2$ and $\varpi$ is a non-zero divisor in each stalk of $\cO_X(D)$.
    The desired property therefore follows from \cite[\href{https://stacks.math.columbia.edu/tag/031Q}{Lemma 031Q}]{stacks-project}.) Since every vector bundle on $\bbP^1_k$ splits, we can find an isomorphism $\pi_\ast \mathcal{L} \cong \cO(i) \oplus \cO(j)$, where $i \leq j$ and (by considering the Euler characteristic) we have $i + j = g-2$. We wish to show that $H^1(\bbP^1_k, \pi_\ast \mathcal{L}) = 0$, or equivalently that $i \geq - 1$. 

    Multiplication by $z$ gives a morphism $\pi_\ast \mathcal{L} \to \cO_{\P^1_k}(g+1) \otimes \pi_\ast \mathcal{L}$, which can be represented as a matrix
    \[ z = \left( \begin{array}{cc} R & S \\ T & Q \end{array}\right) \in H^0\left( \begin{array}{cc} \cO_{\P^1_k}(g+1) & \cO_{\P^1_k}(i-j+g+1) \\ \cO_{\P^1_k}(j-i+g+1) & \cO_{\P^1_k}(g+1) \end{array}\right). \]
    Since $i + j = g-2$, we have $i-j+g+1 = 2i+3$. 

    Suppose for contradiction that $i \leq -2$. Then $2i+3 < 0$, so $S = 0$. We can then calculate
    \[ z^2 = \left( \begin{array}{cc} R^2 & 0 \\ (R+Q)T & Q^2 \end{array}\right) = 0, \]
    hence $R = Q = 0$. Since $H^0(X_k, \mathcal{L}) = H^0(\bbP^1_k, 0 \oplus \cO(j))$, the matrix equality $z = \begin{psmallmatrix} 0 & 0 \\ T & 0 \end{psmallmatrix}$ shows that $z$ annihilates $H^0(X_k, \mathcal{L})$. 
    
    Let $\overline{\eta}$ denote the generic point of $X_k$. Then $\mathcal{L}_{\overline{\eta}}$ is free of rank 1 over $\cO_{X_k, \overline{\eta}}$ (since $\cO_X(D)$ is locally free in a neighbourhood of $\overline{\eta}$), and generated by the image of the global section $1 \in H^0(X, \cO_X(D))$. Since $\mathcal{L}_{\overline{\eta}}$ is not annihilated by $z$, this is a contradiction.
\end{proof}
We have now completed the proof of Theorem \ref{thm_existence_of_integral_representatives}, and have in fact proved the following more precise statement:
\begin{theorem}\label{thm_tautological_section_is_saturated}
    With assumptions and notation as above, there exists an $A$-basis $\mathcal{B} = \{ b_1, \dots, b_n \}$ for $M$ with the following properties:
    \begin{enumerate}
        \item Let $e = b_1 \wedge \dots \wedge b_n$. Then $(M_K, \varphi_K, e)$ is an $f$-module, corresponding under the bijection of Corollary \ref{cor_base_point_description_of_rational_orbits} to the pair $(\alpha, u_k^{-(g+1)} f_0^{n-1} w^{-1}) \in L_f^\times \times K^\times$.
        \item The matrices of the symmetric bilinear forms $\mathrm{ev}_y \circ \varphi_K$, $\mathrm{ev}_x \circ \varphi_K$ with respect to $\mathcal{B}$ have coefficients in $A$, and satisfy $\disc_e(x \mathrm{ev}_y - y \mathrm{ev}_x) = f(x, y)$.
        \item Let $\overline{1} \in M$ denote the image of the section $1 \in H^0(X, \cO_X(D))$ in $M = H^0(S_f, j^\ast \cO_X(D))$. Then $\overline{1}$ is saturated, i.e.\ generates a saturated $A$-submodule of $M$. 
    \end{enumerate}
\end{theorem}
\begin{proof}
    Only the last part remains to be proved. We can again localise and assume that $A$ is a DVR of uniformizer $\varpi$, and must show that $\overline{1} \not\in \varpi M$. We have constructed a pairing $M \times M \to I_f = H^0(S_f, \cO_{S_f}(2g-1))$, and the pairing of $\overline{1}$ with itself is $U(x, y)$, which lies in the image of $H^0(\bbP^1_A, \cO_{\bbP^1_A}(2g-1))$. This image is saturated, and $U(x, y)$ has unit content, by definition, so $U(x, y)$ is saturated in $I_f$. If $\overline{1}$ were a multiple of $\varpi$ then $U(x, y)$ would be a multiple of $\varpi^2$, which is not the case. 
\end{proof}

Let us now take $A = \mathbb{Z}$ and work out the consequences of all the theory developed so far. Fix a polynomial $f(x, y) = f_0 x^n + \dots + f_n y^n \in \mathbb{Z}[x, y]$ of nonzero discriminant. Let $X$ be the corresponding hyperelliptic curve over $\mathbb{Z}$, and let $D_\mathbb{Q}$ be an effective divisor on $X_\mathbb{Q}$ of degree $2g-1$ as at the beginning of \S \ref{sec_integral_orbit_representatives} (therefore assumed not to intersect the locus $z = 0$). Let's assume further that $f_0 f_n \neq 0$. We choose a primitive polynomial $U(x, y) \in \mathbb{Z}[x, y]$ vanishing precisely at the points of $\pi(D_\mathbb{Q})$, with multiplicity; it is determined up to sign. 

Let $M_\mathbb{Q} = H^0(S_{f, \mathbb{Q}}, j_\mathbb{Q}^\ast \cO_{X_\mathbb{Q}}(D_\mathbb{Q}))$; then there is a natural isomorphism $\varphi_\mathbb{Q} : M_\mathbb{Q} \otimes_{L_f} M_\mathbb{Q} \to I_{f, \mathbb{Q}}$ of $L_f$-modules, satisfying $\varphi_\mathbb{Q}(\overline{1} \otimes \overline{1}) = U$, where $\overline{1} \in M_\mathbb{Q}$ is the restriction of the tautological section $1 \in H^0(X_\mathbb{Q}, \cO_{X_\mathbb{Q}}(D_\mathbb{Q}))$.

The following result is a consequence of Theorem \ref{thm_tautological_section_is_saturated}:
\begin{proposition}\label{prop_prepnormof1computation}
    There exists a $\mathbb{Z}$-lattice $M \leq M_\mathbb{Q}$ with the following properties:
    \begin{enumerate}
        \item There exists a $\mathbb{Z}$-basis of $M$ with determinant $e$ such that $(M_\mathbb{Q}, \varphi_\mathbb{Q}, e)$ is an $f$-module and $\varphi(M \otimes M) \leq I_f$.
        \item $\overline{1} \in M$ is saturated.
    \end{enumerate}
\end{proposition}
The $f$-module structure on $M_{\Q}$ determines two symmetric bilinear forms $\mathrm{ev}_y\circ \varphi$ and $\mathrm{ev}_x \circ \varphi$ on $M_{\Q}$.
Associated with these forms is the reduction covariant $H$ on $M_{\R}$ from \S\ref{subsec: reduction covariant def}.
We now compute the norm of the vector $\overline{1} \in M$ with respect to the reduction covariant. Let $X = x / y \in L_f$, and let $\omega_1, \dots, \omega_n \in \mathbb{C}$ be the roots of $f(X, 1)$, which are nonzero by assumption.
Order them so that $|\omega_1|, \dots, |\omega_r|\leq 1$ and $|\omega_{r+1}|, \dots, |\omega_n|> 1$.
Write $f_x$ for the partial derivative of $f$ with respect to $x$, and define $f_y$ similarly.
\begin{proposition}\label{prop_normof1}
    We have
    \begin{align}\label{equation_normof1}
     (\bar{1},\bar{1})_H = \sum_{i=1}^r \frac{| U(\omega_i, 1) |}{| f_x(\omega_i, 1)|} + \sum_{i=r+1}^n \frac{| U(1, \omega_i^{-1}) |}{| f_y(1, \omega_i^{-1})|}.   
    \end{align}
\end{proposition}
\begin{proof}
    Let  $e_i \in L_f \otimes \mathbb{C}$ be the idempotent corresponding to the root $\omega_i$.
    Since $\bar{1}$ is an $L_f$-module generator, $e_1 \bar{1}, \dots,e_n\bar{1}$ is a $\mathbb{C}$-basis of $M_{\mathbb{C}}$. 
    Moreover since the complexified form $\varphi\colon M_{\mathbb{C}}\times M_{\mathbb{C}}\rightarrow L_f\otimes \mathbb{C}$ is $L_f\otimes \mathbb{C}$-bilinear, $\varphi(e_i \bar{1},e_j\bar{1}) = 0$ if $i\neq j$, hence the bilinear forms $\mathrm{ev}_x\circ \varphi$ and $\mathrm{ev}_y \circ \varphi$ have diagonal Gram matrices with respect to the basis $e_1\bar{1} , \dots, e_ n\bar{1}$.
    By Construction \ref{con_reduction_covariant}, we have
    \[ (\overline{1},\overline{1})_H = \sum_{i=1}^n \max( | \mathrm{ev}_y(\varphi(e_i\bar{1}, e_i\bar{1}))|, | \mathrm{ev}_x(\varphi(e_i\bar{1}, e_i\bar{1})) | ).   \]
    We work out the summands explicitly. 
    For $1\leq i \leq n$ we have $\varphi(e_i \bar{1}, e_i \bar{1}) = e_i \varphi(\bar{1}) = e_i U$ by the $L_f$-bilinearity and the explicit description of $\varphi$ from \eqref{eqn_isomorphism_of_line_bundles_in_generic_fibre}.
    We have $\mathrm{ev}_y(e_i U) = \zeta(ye_i U)$ by Lemma \ref{lemma_identificationconnectingmaps}. 
    Using Lemma \ref{lem_computation_of_zeta} and computing as in the proof of Corollary \ref{cor_base_point_description_of_rational_orbits}, we have $\zeta(y e_i U) = f_0^{-1} \tau( e_i U(X, 1) ) = U(\omega_i, 1) / f_x(\omega_i, 1)$.
    Similarly we have $\mathrm{ev}_x(\varphi(e_i\bar{1},e_i\bar{1})) = \zeta(x e_i U) = \zeta(( x / y ) y e_i U) = \omega_i \zeta(y e_i U) = U(1,\omega_i^{-1})/f_y(1,\omega_i^{-1})$. It follows that
\[ \max( | \zeta(y e_i U) |, | \zeta(x e_i U) | ) = \left\{ \begin{array}{cc} | \zeta(y e_i U) | & | \omega_i | \leq 1
\\ | \zeta(x e_i U) | & | \omega_i | \geq 1. \end{array} \right. \]
If we order the roots so that $| \omega_1 |, \dots, | \omega_r | \leq 1$ and $ | \omega_{r+1} |, \dots, |\omega_n | > 1$, then we find
\[ (\overline{1},\overline{1})_H = \sum_{i=1}^r \frac{| U(\omega_i, 1) |}{| f_x(\omega_i, 1)|} + \sum_{i=r+1}^n \frac{| U(1, \omega_i^{-1}) |}{| f_y(1, \omega_i^{-1})|}.  \]
\end{proof}

\section{Equidistribution of the reduction covariant}\label{sec: equidistribution}

The purpose of this section is to show that the reduction covariant \eqref{equation: reduction covariant integral orbits} becomes equidistributed over integral orbits with respect to the natural measure on the space of lattices $\SL_n(\Z)\backslash X$. 
This will follow from a modification of the geometry-of-numbers methods of \cite[\S4]{bhargava2015mosthyperellipticarepointless}. 
We closely follow the structure of \cite[\S3]{lagathorne2024smallheightoddhyperelliptic}.

\subsection{Preliminaries}\label{subsec_preliminariescounting}

Let $n=2g+2$ for some integer $g\geq 1$ and consider the action of $\SL(W) = \SL_n(\Z)$ on the representation $V$ from \ref{subsec_basicdefs}.

\paragraph{Coordinates on $\SL_n(\R)$.}

Let $H_0$ be the standard inner product on $W_{\R} = \R^n$, whose Gram matrix is the identity matrix.
Then $K = \SO_{H_0}(\R) = \SO_n(\R)$ is a maximal compact subgroup of $\SL_n(\R)$.
Let $T  \leq \SL_n$ be the subgroup of diagonal matrices.
For ease of notation, we will write $(t_1,\dots,t_n)$ for the element $\text{diag}(t_1,\dots,t_n)$ of $T$.
Let $N \leq \SL_n$ denote the subgroup of unipotent upper triangular matrices. 
By the Iwasawa decomposition, the product map $N(\R)\times T(\R)^{\circ} \times K \rightarrow \SL_n(\R)$ is a diffeomorphism.

\paragraph{Fundamental set for $\SL_n(\Z)\backslash \SL_n(\R)$.}
We will use the language of semialgebraic sets, see \cite[Chapter 2]{BCR-realalgebraicgeometry}. We will use without further mention that semialgebraic sets are closed under finite unions, intersections, and (pre-)images under semialgebraic maps.

By definition, a Siegel set is a subset of $\SL_n(\R)$ of the form $\omega \cdot T_c \cdot K$, where $\omega \subset N(\R)$ is a relatively compact subset, $c \in \R_{>0}$ and $T_c \coloneqq  \{(t_1,\dots,t_n) \in T(\R)^{\circ} \colon t_1/t_2> c, \dots,t_{n-1}/t_n > c\}$.
(This definition of $T_c$ differs from the one given in \cite[\S4.1.2]{bhargava2015mosthyperellipticarepointless} since our convention for the action of $\SL_n$ on $V$ is slightly different, see \S\ref{subsec_basicdefs}.)
For every such Siegel set, the set of $\gamma\in \SL_n(\Z)$ with $\gamma\cdot  \Siegel \cap \Siegel \neq \varnothing$ is finite \cite[Corollaire 15.3]{Borel-introductiongroupesarithmetiques} (`Siegel property'). 
Since $\SL_n$ is a Chevalley group, there exists a Siegel set $\Siegel$ with the property that $\SL_n(\Z)\cdot \Siegel = \SL_n(\R)$. Explicitly, by \cite[Chapter 4, Theorem 4.12]{PlatonovRapinchuk-Alggroupsandnumbertheory}, we can take any Siegel set with $c\leq 2/\sqrt{3}$ and $\omega$ containing all $n \in N(\R)$ whose off-diagonal entries $n_{ij}$ satisfy $|n_{ij}|\leq 1/2$.
Fix such a $\Siegel$. After enlarging $\Siegel$, we may and do assume that $\omega$, and consequently $\Siegel$, is open and semialgebraic.

The set $\Siegel$ will serve as our fundamental set for the left action of $\SL_n(\Z)$ on $\SL_n(\R)$.
An $\SL_n(\Z)$-orbit might be represented more than once in $\Siegel$, but this does not cause any problems as long as we incorporate the multiplicity function $m \colon \Siegel \rightarrow \Z_{\geq 1}$, defined by $m(x)=\#(\SL_n(\Z) \cdot x \cap \Siegel)$.
This function is bounded and has semialgebraic fibres.

\paragraph{Measures on $\SL_n(\R)$ and $X$.}

Choose a generator of the rank-$1$ module of left invariant top differential forms of $\SL_n$ over $\Z$.
This generator is unique up to sign and induces a bi-invariant Haar measure $dg$ on $\SL_n(\R)$.
We equip the maximal compact subgroup $K$ with its probability Haar measure.
We now explain how these measures also induce measures on $X$ and $\SL_n(\Z)\backslash X$.

The standard inner product $H_0$ defines an element of $X$ and the map $g\mapsto g\cdot H_0$ induces a $\SL_n(\R)$-equivariant bijection $\SL_n(\R)/K \simeq X$.
We will use this identification without further mention.
Since the measure $dg$ on $\SL_n(\R)$ is bi-invariant, it induces measures on $\SL_n(\R)/K=X$ and $\SL_n(\Z)\backslash \SL_n(\R) /K=\SL_n(\Z)\backslash X $. 
We denote the latter measure by $\mu$.
It is a standard fact that $\mu(\SL_n(\Z)\backslash X)$ is finite; in fact, we have $\mu(\SL_n(\Z)\backslash X )= \Vol(\SL_n(\Z) \backslash \SL_n(\R))= \zeta(2)\cdots\zeta(n)$.

\paragraph{Good subsets of $X$.}
Consider the surjective quotient map 
$\varphi\colon \Siegel \rightarrow \SL_n(\Z)\backslash \SL_n(\R) /K = \SL_n(\Z)\backslash X$.
We call a subset $U \subset \SL_n(\Z)\backslash X$ \emph{good} if it is relatively compact and $\varphi^{-1}(U)\subset \Siegel$ is a semialgebraic subset of $\SL_n(\R)$.
For example, the image under $\varphi$ of a relatively compact semialgebraic subset of $\Siegel$ is good.

\begin{lemma}\label{lemma_countablebasis_goodsubsets}
    There exists a countable basis of good open subsets of $\SL_n(\Z)\backslash X$.
\end{lemma}
\begin{proof}
    The proof is identical to that of \cite[Lemma 3.2]{lagathorne2024smallheightoddhyperelliptic}.
\end{proof}

The next two lemmas ensure that good subsets have good properties for the purposes of the geometry-of-numbers arguments.

\begin{lemma}\label{lemma_integral_siegelset}
    Let $U\subset \SL_n(\Z)\backslash X$ be good and let $\bar{U}\subset \SL_n(\R)$ be its preimage under the quotient map $\SL_n(\R)\rightarrow \SL_n(\Z)\backslash X$. 
    Then for all $h\in \SL_n(\R)$ we have
    \begin{align*}
    \int_{g \in \Siegel \cap \bar{U}h} m(g)^{-1}\,dg = \mu(U).
    \end{align*}
\end{lemma}
\begin{proof}
    Follows from pushing forward measures and the bi-invariance of $dg$.
\end{proof}

\begin{lemma}\label{lemma_goodsubsetimpliesregionsemialg}
    Let $U\subset \SL_n(\Z)\backslash X$ be good, and let $A\subset \SL_n(\R)$ be a compact semialgebraic subset. 
    Let $\bar{U}$ be the preimage of $U$ under the quotient map $\SL_n(\R) \rightarrow \SL_n(\Z)\backslash X$.
    Then the set $\{ (g,h) \in \Siegel \times A \colon gh \in \bar{U} \}$ is a semialgebraic subset of $\SL_n(\R)\times \SL_n(\R)$.
\end{lemma}
\begin{proof}
    Identical to the proof of \cite[Lemma 3.4]{lagathorne2024smallheightoddhyperelliptic}.
\end{proof}

\paragraph{Fundamental sets for $\SL_n(\R)\backslash V(\R)$.}

Let $B^s(\R)$ denote the set of polynomials $f = f_0x^{n} +f_1x^{n-1}y + \cdots +f_ny^n \in \R[x]$ with nonzero discriminant. 
Let $I(m) \subset B^s(\R)$ be the subset of polynomials having exactly $2m$ real roots. 
The sets $I(m)$ with $0\leq m\leq n/2$ are the connected components of $B^s(\R)$.
Bhargava \cite[\S4.1.1]{bhargava2015mosthyperellipticarepointless} has partitioned the set of elements in $V(\R)$ whose invariant form lies in $I(m)$ into components $V^{(m,\tau)}$ indexed by $1\leq \tau \leq \lfloor 2^{2m-1}\rfloor$.
Moreover, he constructs for each $\tau$ a fundamental set $L^{(m,\tau)}$ for the action of $\SL_n^{\pm}(\R) = \{g\in \GL_n(\R) \colon \det(g) = \pm 1\}$ on height-$1$ elements in $V^{(m,\tau)}$.
This fundamental set is bounded and semialgebraic and has the property that the invariant binary form map $L^{(m,\tau)}\rightarrow \{ f\in I(m)\colon \height(f) =1 \}$ is a semialgebraic isomorphism.
We do not recall the precise definition of $L^{(m,\tau)}$ here, but we only mention that every element of $L^{(m,\tau)}$ is an $\R_{>0}$-multiple of an element $(A,B)$, where $A$ is the block matrix of the form
\begin{align}\label{eq1_fundset}
    \left(\lambda_1, \cdots , \lambda_{2m}, \psi(\lambda_{2m+1}), \cdots, \psi(\lambda_{n/2+m})\right)
\end{align}
and $B$ is the block matrix of the form 
\begin{align}\label{eq2_fundset}
    \left(\mu_1, \cdots , \mu_{2m}, \psi(\mu_{2m+1}), \cdots, \psi(\mu_{n/2+m})\right),
\end{align}
where $\psi(x+y\sqrt{-1}) = \begin{pmatrix}x & y \\ y& -x \end{pmatrix}$, $\lambda_i, \mu_i \in \R$ for all $1\leq i \leq 2m$, $\lambda_i, \mu_i\in \mathbb{C}$ for all $2m+1\leq i\leq n/2+m$ and $\max(|\lambda_i|, |\mu_i|) = 1$ for all $1\leq i \leq n/2+m$.
(We note that there is a typo in \cite[\S4.1.1]{bhargava2015mosthyperellipticarepointless} in the definition of $\psi$ since the displayed matrix there is not symmetric.)

Since we will be working with the group $\SL_n(\R)$, we need to slightly modify the definition of $L^{(m,\tau)}$ when $m=0$.
Define $L^{(0,1+)} = L^{(0,1)}$ and $L^{(0,1-)} = g\cdot L^{(0,1)}$, where $g$ is the diagonal matrix $(-1,1,\cdots,1)$. 
If $\tau = 1+$ (respectively $\tau=1 -$), let $V^{(0,\tau)}$ denote the set of those $(A,B)\in V(\R)$ such that $f_{A,B}\in I(0)$ and the pair $(\alpha, z) \in L_f^{\times} \times \R^{\times}$ corresponding to the $\SL_n(\R)$-orbit of $(A,B)$ under Corollary \ref{cor_base_point_description_of_rational_orbits} has the property that $z>0$ (respectively $z<0$).
We have $V^{(0,1)} = V^{(0,1+)} \sqcup V^{(0,1-)}$.
Define the indexing set
\begin{align}\label{eq_indexingsetsigma}
    \Sigma = \{(0,1+),(0,1-)\} \cup \{(m,\tau) \colon 1\leq m\leq n/2 \text{ and } 1\leq \tau\leq 2^{2m-1}\}.
\end{align}

\begin{lemma}\label{lemma: Lmtau is fundamental set}
    For each $(m,\tau)\in\Sigma$, $L^{(m,\tau)}$ is a fundamental set for the action of $\SL_n(\R)$ on height-$1$ elements in $V^{(m,\tau)}$.
\end{lemma}
\begin{proof}
    Let $g\in \SL_n^{\pm}(\R)$ be the diagonal matrix $(-1,1,\cdots,1)$.
    An orbit $\SL_n^{\pm}(\R)\cdot v\subset V(\R)$ breaks up into one or two $\SL_n(\R)$-orbits, depending on whether $v$ is $\SL_n(\R)$-conjugate to $g\cdot v$ or not.
    The explicit description of $L^{(m,\tau)}$ shows that if $m>0$ then $g\cdot v = v$ for all $v\in L^{(m,\tau)}$, and so $L^{(m,\tau)}$ is a fundamental set for the $\SL_n(\R)$-action on height $1$ elements in $V^{(m,\tau)}$ since it was such a set for the $\SL_n^{\pm}(\R)$-action.
    If $m=0$ and $v\in L^{(0,1+)}$, then $g\cdot v\in L^{(0,1-)}$ is not $\SL_n(\R)$-conjugate to $v$ by Corollary \ref{cor_base_point_description_of_rational_orbits}, so each $\SL_n^{\pm}(\R)$-orbit in $V^{(0,1)}$ breaks up into two $\SL_n(\R)$-orbits, corresponding to the decomposition $V^{(0,1)} = V^{(0,1+)} \sqcup V^{(0,1-)}$.
\end{proof}

\begin{lemma}\label{lemma:realstabilizers}
    If $(m,\tau)\in \Sigma$ and $v\in V^{(m,\tau)}$, then $\#\Stab_{\SL_n(\R)}(v)$ equals $2^{n/2+m-1}$ if $m>0$ and $2^{n/2}$ if $m=0$.
\end{lemma}
\begin{proof}
    The index of $\Stab_{\SL_n(\R)}(v)$ in $\Stab_{\SL_n^{\pm}(\R)}(v)$ is $2$ or $1$, depending on whether $\SL_n^{\pm}(\R)\cdot v$ breaks up into one or two $\SL_n(\R)$-orbits. 
    This lemma therefore follows from (the proof of) Lemma \ref{lemma: Lmtau is fundamental set} and the computation of $\#\Stab_{\SL_n^{\pm}(\R)}(v)$ in \cite[\S3.2]{bhargava2015mosthyperellipticarepointless}.
    \end{proof}

The next crucial lemma shows that the reduction covariant of \S\ref{subsec: reduction covariant def} is well behaved (in fact, constant) on the fundamental sets $L^{(m,\tau)}$.
\begin{lemma}\label{lem_reduction_covariant_constant_on_section}
    For every pair $(m,\tau)$ as above and elements $(A,B)$, $(A',B')\in L^{(m,\tau)}$, we have $[H_{A,B}] = [H_{(A',B')}]$.
\end{lemma}
\begin{proof}
    This follows from the explicit description of the elements in $L^{(m,\tau)}$ given by the expressions \eqref{eq1_fundset}, \eqref{eq2_fundset}.
    If $m=n/2$, then every pair $(A,B) \in L^{(m,\tau)}$ consists of diagonal matrices, and the condition $\max(|\lambda_i|, |\mu_i|)=1$ above implies that $H_{A,B}$ equals (up to $\R_{>0}$-scaling) the standard inner product on $W_{\R} = \R^n$.
    The case that $m<n/2$ is similar; we omit the details.
\end{proof}

Since we are free to replace $L^{(m,\tau)}$ by $g\cdot L^{(m,\tau)}$ for some $g\in \SL_n(\R)$ and preserve all of its required properties, we may and do assume that the $L^{(m,\tau)}$ have been chosen so that the reduction covariant of every element in $L^{(m,\tau)}$ equals the standard inner product $H_0 \in X$.

\paragraph{Counting lattice points.}

We recall the following proposition \cite[Theorem 1.3]{BarroeroWidmer-lattice}, which strengthens a well-known result of Davenport \cite{Davenport-onaresultofLipschitz}.
\begin{proposition}\label{prop_countlatticepointsbarroero}
	Let $m,n\geq 1$ be integers, and let $Z\subset \R^{m+n}$ be a semialgebraic subset. 
	For $T\in \R^m$, let $Z_T = \{x\in \R^n\colon (T,x) \in Z\}$, and suppose that all such subsets $Z_T$ are bounded.
	Then
	\begin{align*}
		\#(Z_T \cap \Z^n) = \Vol(Z_T)+O(\max\{1,\Vol(Z_{T,j})\}),
	\end{align*}
	where $Z_{T,j}$ runs over all orthogonal projections of $Z_T$ to all $j$-dimensional coordinate hyperplanes $(1\leq j \leq n-1)$. 
	Moreover, the implied constant depends only on $Z$. 
 \end{proposition}

\subsection{The equidistribution theorem}\label{subsec_equidistributiontheorem}

Define the height of a binary form $f(x,y) = f_0x^n + \cdots + f_n y^n\in \R[x,y]$ by the formula $\height(f) = \max |f_i|$.
Let $\mathcal{F}(X)$ be the set of integral binary forms $f = f_0x^{n} + f_1x^{n-1}y + \cdots + f_{n}y^n \in \Z[x,y]$ of nonzero discriminant and of height $<X$. 
\begin{lemma}\label{lemma_F(X)_expectedsize}
    $\# \mathcal{F}(X) = 2^{n+1} X^{n+1} + O(X^{n})$.
\end{lemma}
\begin{proof}
    We need to show that there are $O(X^{n})$ polynomials with $\height(f) < X$ and $\Delta(f) = 0$. This follows immediately from \cite[Lemma 3.1]{bhargava2014geometric}.
\end{proof}

Define the height of an element $(A,B)\in V(\R)$ to be the height of its invariant binary form: $\height(A,B) = \height(f_{A,B})$.
We say an element $(A,B)\in V(\Z)$ is degenerate if $\Delta(A,B) = \disc(f_{A,B})=0$, and nondegenerate otherwise.
For any subset $S\subset V(\Z)$, write $S^{\text{nd}}$ for the subset of nondegenerate elements.
Given $X\in \R_{> 0}$ and an $\SL_n(\Z)$-invariant subset $S\subset V(\Z)$, let $N(S;X)$ be the number of nondegenerate $\SL_n(\Z)$-orbits in $S$ of height $< X$.

Bhargava \cite[\S4,Theorem 8]{bhargava2015mosthyperellipticarepointless} has determined asymptotics for the number of nondegenerate orbits in $V(\Z)$ of bounded height under the action of $\SL_n^{\pm}(\Z) = \GL_n(\Z)$.
The same computation applies to nondegenerate $\SL_n(\Z)$-orbits, showing that there exists a constant $c>0$ such that 
\begin{align}\label{eq_asymptoticnumberofintegralorbits}
    N(V(\Z);X) = c X^{n+1} + o(X^{n+1})
\end{align}
as $X\rightarrow +\infty$.
The purpose of this section is to show, in analogy with \cite[Theorem 3.8]{lagathorne2024smallheightoddhyperelliptic}, that the reduction covariant \eqref{equation: reduction covariant integral orbits} equidistributes over nondegenerate orbits, in the following sense:

\begin{theorem}\label{theorem: equidistribution reduction covariant}
    Let $U\subset \SL_n(\Z)\backslash X$ be a Borel-measurable subset whose boundary has measure zero.
    Then 
    \begin{align*}
        N(V(\Z)\cap \mathcal{R}^{-1}(U); X) = \frac{\mu(U)}{\mu(\SL_n(\Z)\backslash X) }\cdot N(V(\Z);X)+ o(X^{n+1}).
    \end{align*}
\end{theorem}
The proof of Theorem \ref{theorem: equidistribution reduction covariant} follows from a modification of the proof of Bhargava of the asymptotic \eqref{eq_asymptoticnumberofintegralorbits} and is given below.
We first prove a local version for good subsets.
In the notation of \S\ref{subsec_preliminariescounting}, fix a pair $(m,\tau)\in \Sigma$ (see \eqref{eq_indexingsetsigma}), which corresponds to a subset $V^{(m,\tau)}\subset V(\R)$, and write $V(\Z)^{(m,\tau)} = V(\Z) \cap V^{(m,\tau)}$.

Let $\mathcal{J}$ be the rational nonzero constant of \cite[Proposition 16]{bhargava2015mosthyperellipticarepointless}.
Let $r_{m,n}$ equal $\frac{1}{2}\#\Stab_{\SL_n(\R)}(v)$ for any $v\in V^{(m,\tau)}$; this constant is computed explicitly in Lemma \ref{lemma:realstabilizers}.
For a subset $S$ of $V(\R)$ or $I(m)$, write $S_{<X}$ for the set of $s\in S$ with $\height(s)<X$.
Equip $I(m)$ with the measure corresponding to the differential form $df_0\wedge \cdots \wedge df_n$.
Let 
\begin{align*}
    c_{m,\tau} = \mu(\SL_n(\Z)\backslash X)\frac{|\mathcal{J}| \cdot  \Vol(I(m)_{<1})}{r_{m,n}},
\end{align*}
where $v$ denotes any element of $V^{(m,\tau)}$.
Bhargava \cite[Theorem 9]{bhargava2015mosthyperellipticarepointless} showed that 
\begin{align}\label{equation_averageorbitslocalversion}
    N(V(\Z)^{(m,\tau)};X) = c_{m,\tau} X^{n+1} + o(X^{n+1}).
\end{align}
\begin{proposition}\label{prop_equidistribution_realcomponent}
    Let $U\subset \SL_n(\Z)\backslash X$ be a good subset and let $S_U \coloneqq V(\Z)^{(m,\tau)} \cap \mathcal{R}^{-1}(U)$.
    Then 
    \begin{align}\label{equation_equidistribution_realcomponent}
    N(S_U;X) = \frac{\mu(U)}{\mu(\SL_n(\Z)\backslash X)} c_{m,\tau} X^{n+1} + o(X^{n+1}).
\end{align}
\end{proposition}
\begin{proof}
    We will use the notations and choices made in \S\ref{subsec_preliminariescounting}.
    Write $\Lambda = \R_{>0}$, equipped with its multiplicative Haar measure $d^{\times}\lambda$, and let $\Lambda$ act on $V(\R)$ by scalar multiplication.
    Write $L  = L^{(m,\tau)}$.
    Fix a compact, semialgebraic set $G_0\subset \SL_n(\R) \times \Lambda$ of volume $1$, with nonempty interior, that satisfies $K\cdot G_0 = G_0$ and whose projection onto $\Lambda$ is contained in $[1,K_0]$ for some $K_0>1$.
    
    For any subset $S\subset V(\Z)^{(m,\tau)}$, averaging over $G_0$ \cite[Equation (17)]{bhargava2015mosthyperellipticarepointless} shows 
    \begin{align*}
    N(S;X) = \frac{1}{r_{m,n}} \int_{h \in G_0} \#[  S \cap (\Siegel \Lambda h L)_{<X}] \,dh,
    \end{align*}
    with the caveat that elements on the right hand side are weighted by a function similar to \cite[\S6.5, Equation (6.5)]{laga-f4paper}.
    We use the above expression to define $N(S;X)$ for subsets $S\subset V(\Z)^{(m,\tau)}$ that are not necessarily $\SL_n(\Z)$-invariant.
    A change of variables trick \cite[Equation (22)]{BhargavaGross} shows that for every $S\subset V(\Z)^{(m,\tau)}$ we then have
    \begin{align*}
    N(S;X) = \frac{1}{r_{m,n}} \int_{g\in \Siegel} \int_{\lambda \in \Lambda}\#[S \cap (g\lambda G_0 L)_{<X}] m(g)^{-1} \,dg \, d^{\times}\lambda,
    \end{align*}
    with the caveat that an element $v\in S\cap (g\lambda G_0 L)$ on the right hand side is counted with multiplicity $\#\{h\in G_0 \colon v\in g\lambda h L \}$.

    We let $V^{\text{cusp}}$ be the subset of all $(A,B)\in V(\R)$ such that the top left entries of $A$ and $B$ have absolute value $<1$.
    By cutting off the cusp arguments \cite[Proposition 13]{bhargava2015mosthyperellipticarepointless}, $N(S_U;X) = N(S_U';X)+ o(X^{n+1})$, so it remains to estimate $N(S_U';X)$.

    We first make the reduction covariant more explicit.
    If $g\in \Siegel, h = (h',\lambda')\in G_0$ (with $h'\in \SL_n(\R)$ and $\lambda'\in \Lambda$), $\lambda \in \Lambda$ and $\ell \in L$, the $\SL_n(\R)$-equivariance and $\Lambda$-invariance of $\mathcal{R}$ imply that $\mathcal{R}(g h \lambda \ell) = g h' \mathcal{R}(\ell)$.
    We have chosen $L$ so that $\mathcal{R}(\ell) = 1\in \SL_n(\R) / K$, so $gh' \mathcal{R}(\ell) = gh'$ in $X$.
    Write $\bar{U}$ for the preimage of $U$ under the quotient map $\SL_n(\R)\rightarrow \SL_n(\Z)\backslash X$.
    Then we conclude that $\mathcal{R}(gh\lambda \ell) \in U$ if and only if $gh' \in \bar{U}$.

    Let $Z_1 = \{ (g,h) \in \Siegel \times G_0 \mid gh' \in \bar{U}\}$, where we write $h'$ for the projection of $h \in G_0$ onto $\SL_n(\R)$.
    By Lemma \ref{lemma_goodsubsetimpliesregionsemialg}, $Z_1$ is semialgebraic.
    Let $Z_2$ be the graph of the action map $Z_1\times \Lambda\times L\rightarrow V(\R)$ sending $(g,h,\lambda,\ell)$ to $gh\lambda v$.
    Since this map is algebraic, $Z_2$ is semialgebraic.
    So is the projection $Z_3$ of $Z_2$ onto $\Siegel\times V(\R)$.
    Let $Z_4 = \{(g,v,X) \in Z_3\times \R_{>0}\mid \height(v)< X\}$; this is again semialgebraic.
    For every $g\in \Siegel$ and $X\in \R_{>0}$ the set $B(g,X)\coloneqq \{v \in V(\R)\mid (g,v,X)\in Z_4\}$ is semialgebraic and equals $(g\Lambda G_0 L)_{<X} \cap \mathcal{R}^{-1}(U)$.
    We view $B(g,X)$ as a multiset where $v\in B(g,X)$ has multiplicity $\#\{(\lambda,h) \in \Lambda\times G_0 \mid v \in g\lambda h L \}$.
    Then $B(g,X)$ is partitioned into finitely many semialgebraic subsets of constant multiplicity.
    Applying Proposition \ref{prop_countlatticepointsbarroero} to $Z_4$ shows that 
    \begin{align*}
        \#[V(\Z) \cap (B(g,X)\setminus V^{\text{cusp}})] = \Vol(B(g,X) \setminus V^{\text{cusp}}) + E(g,X),
    \end{align*}
    where $E(g,X)$ is the error term. 
    The proof now proceeds in an identical way to that of \cite[Proposition 15]{bhargava2015mosthyperellipticarepointless}: estimating $E(g,X)$ and $\Vol(B(g,X)\cap V^{\text{cusp}})$ shows that
    \begin{align}\label{eq_proofequidist2}
        N(S_U;X) = \frac{1}{r_{m,n}} \int_{g\in \Siegel} \Vol(B(g,X))m(g)^{-1} \, dg + o(X^{n+1}).
    \end{align}
    The change of measure formula of \cite[Proposition 16]{bhargava2015mosthyperellipticarepointless} shows that 
    \begin{align}\label{eq_proofequidist3}
     \frac{1}{2^{m+n}} \int_{g\in \Siegel} \Vol(B(g,X))m(g)^{-1} \,dg = \frac{|\mathcal{J}|}{2^{m+n}} \Vol(I(m)_{<X})\int_{g \in \Siegel} \int_{h\in G_0}  \mathbf{1}_{\{gh' \in \bar{U}\}} m(g)^{-1} \,dh \,dg,
    \end{align}
    where $\mathbf{1}_T$ denotes the indicator function of a set $T$, and where $h'$ denotes the $\SL_n(\R)$-component of $h$.
    By switching the order of integration, Lemma \ref{lemma_integral_siegelset} and using $\Vol(G_0) = 1$, we calculate that
    \begin{align}\label{eq_proofequidist4}
        \int_{g \in \Siegel} \int_{h\in G_0}  \mathbf{1}_{\{gh' \in \bar{U}\}}m(g)^{-1} \, dh \, dg = \int_{h\in G_0} \int_{g\in \Siegel} \mathbf{1}_{\bar{U}h'^{-1}}m(g)^{-1} \,  dg \, dh = \int_{h\in G_0} \mu(U) \, dh= \mu(U).
    \end{align}
    Combining \eqref{eq_proofequidist2}, \eqref{eq_proofequidist3} and \eqref{eq_proofequidist4} shows that
    \begin{align*}
    N(S_U;X) = \mu(U)\frac{|\mathcal{J}|}{r_{m,n}} \Vol(I(m)_{<X}) +o(X^{n+1})=  \mu(U)\frac{|\mathcal{J}|}{r_{m,n}} \Vol(I(m)_{<1})X^{n+1}+o(X^{n+1}),
    \end{align*}
    as required.
\end{proof}

\begin{proof}[Proof of Theorem \ref{theorem: equidistribution reduction covariant}]
    Let 
    \[
    \underline{\nu}(U) = \mu(\SL_n(\Z)\backslash X) \liminf_{X\rightarrow \infty} \frac{N(V(\Z) \cap \mathcal{R}^{-1}(U);X)}{N(V(\Z);X)}
    \]
    and $\bar{\nu}(U)$ be the same expression with $\liminf$ replaced by $\limsup$.
    It suffices to prove that $\underline{\nu}(U) = \bar{\nu}(U) = \mu(U)$.
    If $U$ is good, this follows from summing \eqref{equation_equidistribution_realcomponent} over all $(m,\tau)\in \Sigma$, Equation \eqref{equation_averageorbitslocalversion}, and the fact that $V^s(\R)$ is partitioned into the subsets $V^{(m,\tau)}$.
    To prove the theorem for general $U$, we bootstrap from the case of good subsets (and the trivial case $U = \SL_n(\Z) \backslash X$).

    Let $U^{\circ}$ be the interior of $U$.
    Since $U^{\circ}$ is open, by Lemma \ref{lemma_countablebasis_goodsubsets} there exists an increasing sequence $(U_n)_{n\geq 1}$ of good subsets whose union is $U^{\circ}$.
    For every $n\geq 1$ we have $\underline{\nu}(U_n) \leq \underline{\nu}(U^{\circ})$.
    Since $U_n$ is good, $\underline{\nu}(U_n) = \mu(U_n)$, hence $\mu(U_n) \leq \underline{\nu}(U^{\circ})$ for all $n\geq 1$.
    By continuity of the measure, $\mu(U_n) \rightarrow \mu(U^{\circ})$ as $n\rightarrow \infty$.
    We conclude that $\mu(U^{\circ})\leq \underline{\nu}(U^{\circ})$.
    Let $\bar{U}$ denote the closure of $U$.
    Since the complement of $\bar{U}$ equals the interior of the complement of $U$, the above argument also shows $\bar{\nu}(\bar{U}) \leq \mu(\bar{U})$.

    In conclusion, we have shown that 
    \[
    \mu(U^{\circ}) \leq \underline{\nu}(U^{\circ})\leq \underline{\nu}(U) \leq \bar{\nu}(U) \leq \bar{\nu}(\bar{U}) \leq \mu(\bar{U}).
    \]
    By assumption, $\mu(U^{\circ}) = \mu(\bar{U}) = \mu(U)$, so all the inequalities are in fact equalities, and the theorem follows.
\end{proof}

\subsection{A neighbourhood of the cusp}

Let $\epsilon >0$.
Say an inner product $H$ on $W_{\R} =\R^n$ has an $\epsilon$-small vector if there exists a nonzero $v\in \Z^n$ such that 
\begin{align*}
    (v,v)^{1/2}_H< \epsilon \cdot (\det H)^{1/2n} . 
\end{align*}
This condition only depends on the image $\SL_n(\Z) \cdot [H]$ of $H$ in $\SL_n(\Z) \backslash X$, so it makes sense say that an element of $\SL_n(\Z)\backslash X$ has an $\epsilon$-small vector.
We now show that such elements form a small measure subset of $\SL_n(\Z)\backslash X$.

Recall from \S\ref{subsec_preliminariescounting} that $\Siegel  = \omega T_c K$ and that $\varphi\colon \Siegel \rightarrow \SL_n(\Z)\backslash X$ denotes the projection.
After possibly enlarging $\Siegel$, we may and do assume that $c<1$.
Let $T(\epsilon) = \{(t_1,\dots,t_n) \in T_c \colon t_1 > c^{n-1}/\epsilon\}$ and $U_{\epsilon} = \varphi(\omega T(\epsilon) K)\subset \SL_n(\Z)\backslash X$.

\begin{proposition}\label{prop_good_neighbourhoods_of_cusp}
    \begin{enumerate}
        \item The set $U_{\epsilon}$ is Borel-measurable, its boundary has measure zero, and $\mu(U_{\epsilon})\rightarrow 0$ as $\epsilon \rightarrow 0$.
        \item If $H$ is an inner product on $W_{\R}$ that admits an $\epsilon$-small vector, then the image of $H$ in $\SL_n(\Z)\backslash X$ lies in $U_{\epsilon}$.
    \end{enumerate}
\end{proposition}
\begin{proof}
    Since the boundary of a semialgebraic set has strictly smaller dimension \cite[Proposition 2.8.13]{BCR-realalgebraicgeometry} and $\varphi^{-1}(U_{\epsilon})$ is semialgebraic, the boundary of $\varphi^{-1}(U_{\epsilon})$ has measure zero, hence the same is true for $U_{\varepsilon}$. To show that $\mu(U_{\epsilon})\rightarrow 0$ as $\epsilon\rightarrow 0$, it suffices to show that $\vol(\omega T(\epsilon) K) \rightarrow 0$ as $\epsilon \rightarrow 0$. This is true because $\vol(\Siegel)<\infty$ and $\cap_{\epsilon >0} (\omega T(\epsilon) K) =0$.

    To prove the second part, we may assume that $H$ has determinant $1$ and so (since $\varphi$ is surjective) that $H = g\cdot H_0 = g^{-t} H_0 g^{-1}$ for some $g =nt$ with $n\in \omega$ and $t = (t_1,\dots, t_n)\in T_c$. 
    Let $v\in W=\Z^n$ be a nonzero element with $(v,v)_H^{1/2} < \epsilon$.
    We will show that this implies $t_1> c^{n-1}/\epsilon$.
    Write $e_1,\dots,e_n$ for the standard basis of $\Z^n$ and let $f_i = g \cdot e_i$ for all $i$. 
    Since $(v,w)_{H} = (g^{-1}v,g^{-1}w)_{H_0}$ for all $v,w\in W_{\R}$, the basis $f_1,\dots,f_n$ is orthonormal with respect to $(-,-)_H$.
    Write $v = \sum_{i=1}^k m_i e_i$ where $m_i$ are integers, $1\leq k\leq m$ and $m_k\neq 0$.
    A computation reveals that $e_i \in t_i^{-1} f_i + \text{span}_{\R}\{f_1,\dots,f_{i-1}\}$ for each $i$, so $v \in m_k t_k^{-1} f_k + \text{span}_{\R}\{f_1,\dots,f_{k-1}\}$.
    By the orthonormality of the $f_i$ and the fact that $m_k$ is a nonzero integer,
    \begin{align*}
        (v,v)^{1/2}_H\geq |m_k|t_k^{-1}\geq t_k^{-1}.
    \end{align*}
    Therefore $t_k > 1/\epsilon$.
    Since $(t_1,\dots,t_n) \in T_c$, we have $t_i > c\cdot  t_{i+1}$ for all $1\leq i \leq k-1$, so $t_1 > c^{k-1} t_k > c^{k-1}/ \epsilon > c^{n-1}/\epsilon$, as required.
\end{proof}

Combining Proposition \ref{prop_good_neighbourhoods_of_cusp} with the equidistribution theorem (Theorem \ref{theorem: equidistribution reduction covariant}) and \eqref{eq_asymptoticnumberofintegralorbits} immediately shows:

\begin{corollary}\label{corollary_polyswithsmallnormvectorrare}
    There exist a constant $c_1>0$ such that for every $\epsilon>0$,
    \begin{align}\label{equation: lim sup less than mu(U)}
        \limsup_{X\rightarrow \infty} \frac{\#\{f\in \mathcal{F}(X) \colon \exists\, (A,B)\in V_f(\Z) \text{ and } w\in W \text{ with } ( w,w)_{H_{A,B}}^{1/2} < \epsilon  (\det H_{A,B})^{1/2n} \} }{ \#\mathcal{F}(X)} \leq c_1\cdot \mu(U_{\epsilon}).
    \end{align}
    Moreover, $\mu(U_{\epsilon})\rightarrow 0$ as $\epsilon \rightarrow 0$.
\end{corollary}

\section{A height lower bound for divisors}\label{sec: a height lower bound for divisors}

\subsection{A density $1$ family}

Let $g \geq 1$. Recall from \S\ref{subsec_equidistributiontheorem} that for a real number $X>0$, $\mathcal{F}(X)$ denotes the set of integral binary forms $f(x,y) \in \Z[x,y]$ of degree $n=2g+2$, nonzero discriminant and height at most $X$, and that $\#\mathcal{F}(X) = (2X)^{n+1} +O(X^n)$. 

For $\delta>0$, let $\mathcal{F}_{\delta}(X)$ be the subset of $f = f_0 x^n + \cdots + f_n y^n\in \mathcal{F}(X)$ satisfying the following properties:
\begin{enumerate}
    \item $f(x,y)$ is an irreducible polynomial in $\Q[x,y]$.
    \item We have $|\disc(f)|\geq X^{2n-2-\delta}$. 
\end{enumerate}
Note that the first condition implies that $f_0f_n \neq 0$.

These conditions cut out a density-$1$ family:

\begin{proposition}\label{prop_subfamilyhasdensity1}
    We have $\mathcal{F}_{\delta}(X) = (2X)^{n+1} + o(X^{n+1})$.
\end{proposition}
\begin{proof}
    We need to show that the number of $f\in \mathcal{F}(X)$ failing the first or second condition is $o(X^{n+1})$.
    For the first condition, this follows from Hilbert's irreducibility theorem; for the second, this is \cite[Lemma 6.1]{BSW-squarefreeII}.
\end{proof}

The relevance of this family comes from the next elementary proposition. 
To state it, let $f\in \mathcal{F}(X)$ and write
\begin{align}\label{eq_realdecompositionpoly}
    f(x,y) = c \prod_{i=1}^r (x-\omega_iy) \prod_{j=1}^k (\eta_jx-y), 
\end{align}
where $c, \omega_i,\eta_j\in \mathbb{C}$, $|\omega_i|\leq 1$ and $|\eta_j|<1$ for all $i,j$.
Write $f_x(x,y)$ for the partial derivative of $f(x,y)$ with respect to $x$, and similarly for $f_y(x,y)$.

Let $U(x,y)  = u_0 x^{m} + u_1x^{m-1} y + \cdots + u_{m}y^m\in \Z[x,y]$ be an integral binary form of degree $m$. 
Write $h(U) = \log(\max(|u_0|, \dots,|u_m|))$.

\begin{proposition}\label{prop_massagingformuladensity1family}
    If $f\in \mathcal{F}_{\delta}(X)$ then, in the above notation, we have
    \begin{align}\label{equation: upper bound formula density 1 family}
        n\cdot \log\left(\sum_{i=1}^r \frac{|U(\omega_i,1)|}{|f_x(\omega_i,1)|}
        +\sum_{j=1}^k \frac{|U(1,\eta_j)|}{|f_y(1,\eta_j)|}\right) - \log |c| 
        \leq n\cdot h(U) - (n+1)\log X +  n \delta \log X + \kappa_{n,m},
    \end{align}
    where $\kappa_{n,m}$ is a constant that only depends on $n$ and $m$.
\end{proposition}
\begin{proof}
    We have the formulae 
    \begin{align*}
        \disc(f) &= c^{2n-2} \prod_{1\leq i<j\leq r} (\omega_i-\omega_j)^2 \prod_{1\leq i<j\leq k}(\eta_i-\eta_j)^2 \prod_{1 \leq i \leq r, 1\leq j\leq k} (\omega_i\eta_j-1)^2 , \\
        f_x(\omega_i,1) &= c\prod_{j\neq i}(\omega_i -\omega_j) \prod_{j=1}^k(\omega_i\eta_j-1) \text{ for each }i=1,\dots,r.
    \end{align*}
    Each term $|\omega_i-\omega_j|, |\eta_i-\eta_j|, |\omega_i\eta_j-1|$ is $\leq 2$. 
    Therefore, $|\disc(f)|/|f_x(\omega_i,1)|$ is $|c|^{2n-3}$ times a product of $n(n-1)-(n-1)$ terms that are each $\leq 2$.
    Hence $|\disc(f)|/|f_x(\omega_i,1)|\leq |c|^{2n-3} 2^{(n-1)^2}$.
    Since $|\disc(f)|\geq X^{2n-2-\delta}$ and $|U(\omega_i,1)|\leq (m+1) \max(|u_0|, \cdots, |u_m|)$, we have 
    \begin{align*}
        \frac{|U(\omega_i,1)|}{|f_x(\omega_i,1)|}\leq (m+1) \max(|u_i|) 2^{(n-1)^2} |c|^{2n-3}/X^{2n-2-\delta}.
    \end{align*}
    The same upper bound applies to $\frac{|U(1,\eta_j)|}{|f_y(1,\eta_j)|}$, and summing these bounds shows that the left hand side of \eqref{equation: upper bound formula density 1 family} is at most
    \begin{align}\label{eq1_proofformuladensity1}
        n\log(n(m+1)2^{(n-1)^2}) + n\cdot h(U) + (n(2n-3)-1)\log |c| - n(2n-2-\delta)\log X.
    \end{align}
    It remains to upper bound $\log |c|$, which we achieve by showing that the elements $\omega_i,\eta_j$ are not too far removed from the unit circle.
    Indeed, \cite[Theorem 1]{soundararajan-equidistributionzeroespolys} shows that 
    \begin{align*}
        \prod_{i=1}^r |\omega_i|^{-1} \prod_{j=1}^k |\eta_j|^{-1} \leq \frac{n}{|f_n/f_0|} \sum_{i=0}^n |f_i/f_0|^2.
    \end{align*}
    On the other hand, $|c|^2 \prod |\omega_i| \prod |\eta_j| = |f_0||f_n|$ by \eqref{eq_realdecompositionpoly}, so 
    \begin{align*}
        |c|^2 = |f_0||f_n|\prod_{i=1}^r |\omega_i|^{-1} \prod_{j=1}^k |\eta_j|^{-1} \leq n \sum_{i=0}^n |f_i|^2 \leq n(n+1) X^{2},
    \end{align*}
    hence $\log|c|\leq \frac{1}{2} \log(n(n+1)) + \log X$.
    Combining the latter upper bound with \eqref{eq1_proofformuladensity1} concludes the proof.
\end{proof}

\subsection{Height functions}

Let $f\in \mathcal{F}(X)$ and let $\pi\colon X_f\rightarrow \P^1_{\Q}$ be the hyperelliptic curve described by the equation $z^2 = f(x,y)$. 
If $D$ is an effective divisor on $X_f$ of degree $m$, let $U(x,y) = u_0 x^m + \cdots +u_m y^m\in \Z[x,y]$ be a primitive binary form of degree $m$ vanishing precisely at the points of $\pi(D)$ with multiplicity. 
Then $U$ is uniquely determined up to sign and we may define
\begin{align*}
    h(D) =h(U) =  \log \max(|u_0|, \dots,|u_m|). 
\end{align*}
This is a `naive height' on the set of effective divisors on $X_f$.
For example, if $D$ has degree $1$, corresponding to a point $P\in X_f(\Q)$, then $h(D)$ equals $h(\pi(P))$, where $h(\alpha)$ denotes the logarithmic normalized Weil height of an element $\alpha \in \P^1(\bar{\Q})$.
More generally, by \cite[Theorem VIII.5.9]{Silverman-arithmeticellcurvesbook} we have:
\begin{lemma}\label{lemma_heightalgebraicpointminpoly}
    If $D$ is the Galois orbit of an algebraic point $P\in X_f(\bar{\Q})$ of degree $m$, then $|h(D) - m h(\pi(P)) | \leq m \log 2$.
\end{lemma}

\begin{theorem}\label{theorem_upperboundvectorlengthusingheight}
    Suppose that $f \in \mathcal{F}_{\delta}(X)$ and that $D$ is an effective divisor on $X_f$ of degree $2g-1$.
    Then there exists $(A,B)\in V_f(\Z)$ and a primitive $w\in W=\Z^n$ such that 
    \begin{align*}
        n\cdot \log(w,w)_{H_{A,B}} - \log \det H_{A,B} \leq n\cdot h(D) - (n+1)\log X + n\delta \log X + \kappa_{n},
    \end{align*}
    where $\kappa_n$ is a constant that only depends on $n$.
\end{theorem}
\begin{proof}
    Since $f = f_0 x^n + \cdots + f_n y^n \in \mathcal{F}_{\delta}(X)$ is irreducible, $f_0f_n \neq 0$ and $D$ does not intersect the irreducible Weierstrass locus $S_f$.
    By Propositions \ref{prop_prepnormof1computation}, \ref{prop_normof1} and Lemma \ref{lemma: determinant of reduction covariant}, there exists an $(A,B)$ and a primitive $w\in W$ such that $n\log (w,w)_{H_{A,B}} - \det H_{A,B}$ equals the left hand side of \eqref{equation: upper bound formula density 1 family}, so we conclude using Proposition \ref{prop_massagingformuladensity1family}.
\end{proof}

\subsection{Proof of the main theorem}

Let $\mathcal{F} = \cup_{X>0} \mathcal{F}(X)$ be the set of all binary forms $f(x,y)\in \Z[x,y]$ of degree $n=2g+2$ and nonzero discriminant, ordered by height. If $S \subset \mathcal{F}$ is a subset and $\alpha \in [0, 1]$, we say that $S$ has density $\alpha$ if
\[ \lim_{X \to \infty} \frac{ \# (S \cap \mathcal{F}(X)) }{ \# \mathcal{F}(X) } = \alpha. \]
\begin{theorem}\label{theorem_lowerboundheightdivisor}
    Let $\epsilon>0$ be arbitrary, and let $S_\epsilon$ denote the set of $f \in \mathcal{F}$ such that every effective divisor $D$ on $X_f$ of odd degree $\leq 2g-1$ satisfies
    \begin{align}\label{eq_heightinequalitydivisor}
        h(D) \geq \left(1 + \frac{1}{2g+2} -\epsilon \right) \log \height f.
    \end{align}
    Then $S_\epsilon$ has density $1$. 
\end{theorem}
\begin{proof}
    Adding the degree-$2$ divisor $\pi^{-1}(\infty)$ to $D$ has the effect of multiplying $U$ by $y^2$, which does not affect the quantity $h(D)$. 
    It therefore suffices to prove that the set $S_\epsilon' \subset S_\epsilon$ of $f\in \mathcal{F}$ satisfying \eqref{eq_heightinequalitydivisor} only for $D$ of degree $2g-1$ has density 1. 
    
    Let $\mathcal{F}^{\text{bad}}$ be the subset of those $f\in \mathcal{F}$ such that there exists a degree $2g-1$ divisor $D$ for which \eqref{eq_heightinequalitydivisor} does \emph{not} hold; it suffices to prove that $\mathcal{F}^{\text{bad}}\subset \mathcal{F}$ has density zero. Fix $\delta \in (0, \epsilon)$. 
    By Proposition \ref{prop_subfamilyhasdensity1}, it suffices to prove that $\#(\mathcal{F}^{\text{bad}} \cap \mathcal{F}_{\delta}(X)) = o(X^{2g+3})$.

    Write $\epsilon_1 = \epsilon - \delta$.
    Let $\mathcal{F}(X)^{\text{short}}$ be the subset of $f\in \mathcal{F}(X)$ for which there exists elements $(A,B) \in V_f(\Z)$ and $w\in W - \{ 0 \}$ that satisfy 
    \begin{align}\label{eq1_proofmaintheorem}
         n\log (w,w)_{H_{A,B}} - \log \det H_{A,B}< -\epsilon_1 \log X.
    \end{align}
    Theorem \ref{theorem_upperboundvectorlengthusingheight} implies that $\mathcal{F}^{\text{bad}} \cap \mathcal{F}_{\delta}(X)\subset \mathcal{F}(X)^{\text{short}}$ for sufficiently large $X$.
    The equidistribution theorem, more specifically Corollary \ref{corollary_polyswithsmallnormvectorrare}, shows that $\#\mathcal{F}(X)^{\text{short}} = o(X^{2g+3})$.
    Therefore $\#(\mathcal{F}^{\text{bad}} \cap \mathcal{F}_{\delta}(X)) = o(X^{2g+3})$, as required.
\end{proof}

This theorem and Lemma \ref{lemma_heightalgebraicpointminpoly} immediately imply:

\begin{corollary}
    Let $\epsilon>0$ be arbitrary. Then for $100\%$ of $f\in \mathcal{F}(X)$, every algebraic point $P\in X_f(\bar{\Q})$ of odd degree $m\leq 2g-1$ satisfies
    \begin{align*}
        m\cdot h(\pi(P)) \geq \left(1 + \frac{1}{2g+2} -\epsilon \right) \log \height(f).
    \end{align*}
\end{corollary}

\bibliographystyle{abbrv}

\begin{thebibliography}{10}

\bibitem{BarroeroWidmer-lattice}
F.~Barroero and M.~Widmer.
\newblock Counting lattice points and {O}-minimal structures.
\newblock {\em Int. Math. Res. Not. IMRN}, (18):4932--4957, 2014.

\bibitem{BayerEisenbud-ribbons}
D.~Bayer and D.~Eisenbud.
\newblock Ribbons and their canonical embeddings.
\newblock {\em Trans. Amer. Math. Soc.}, 347(3):719--756, 1995.

\bibitem{bhargava2014geometric}
M.~Bhargava.
\newblock The geometric sieve and the density of squarefree values of invariant
  polynomials.
\newblock arXiv preprint, \texttt{1402.0031v1}, 2014.

\bibitem{Bhargava-hasseprincipleplanecubics}
M.~Bhargava.
\newblock A positive proportion of plane cubics fail the {H}asse principle.
\newblock arXiv preprint, \texttt{1402.1131v1}, 2014.

\bibitem{bhargava2015mosthyperellipticarepointless}
M.~Bhargava.
\newblock Most hyperelliptic curves over {$\mathbb{Q}$} have no rational
  points.
\newblock arXiv preprint, \texttt{1308.0395v1}, 2015.

\bibitem{BhargavaGross}
M.~Bhargava and B.~H. Gross.
\newblock The average size of the 2-{S}elmer group of {J}acobians of
  hyperelliptic curves having a rational {W}eierstrass point.
\newblock In {\em Automorphic representations and {$L$}-functions}, volume~22
  of {\em Tata Inst. Fundam. Res. Stud. Math.}, pages 23--91. Tata Inst. Fund.
  Res., Mumbai, 2013.

\bibitem{BGW15}
M.~Bhargava, B.~H. Gross, and X.~Wang.
\newblock Arithmetic invariant theory {II}: {P}ure inner forms and obstructions
  to the existence of orbits.
\newblock In {\em Representations of reductive groups}, volume 312 of {\em
  Progr. Math.}, pages 139--171. Birkh\"{a}user/Springer, Cham, 2015.

\bibitem{BGW17}
M.~Bhargava, B.~H. Gross, and X.~Wang.
\newblock A positive proportion of locally soluble hyperelliptic curves over
  {$\Bbb Q$} have no point over any odd degree extension.
\newblock {\em J. Amer. Math. Soc.}, 30(2):451--493, 2017.
\newblock With an appendix by Tim Dokchitser and Vladimir Dokchitser.

\bibitem{Bhargava-squarefree}
M.~Bhargava, A.~Shankar, and X.~Wang.
\newblock Squarefree values of polynomial discriminants {I}.
\newblock {\em Invent. Math.}, 228(3):1037--1073, 2022.

\bibitem{BSW-squarefreeII}
M.~Bhargava, A.~Shankar, and X.~Wang.
\newblock Squarefree values of polynomial discriminants {II}.
\newblock arXiv preprint, \texttt{2207.05592v1}, 2022.

\bibitem{BCR-realalgebraicgeometry}
J.~Bochnak, M.~Coste, and M.-F. Roy.
\newblock {\em Real algebraic geometry}, volume~36 of {\em Ergebnisse der
  Mathematik und ihrer Grenzgebiete (3) [Results in Mathematics and Related
  Areas (3)]}.
\newblock Springer-Verlag, Berlin, 1998.
\newblock Translated from the 1987 French original, Revised by the authors.

\bibitem{Borel-introductiongroupesarithmetiques}
A.~Borel.
\newblock {\em Introduction aux groupes arithm\'{e}tiques}.
\newblock Publications de l'Institut de Math\'{e}matique de l'Universit\'{e} de
  Strasbourg, XV. Actualit\'{e}s Scientifiques et Industrielles, No. 1341.
  Hermann, Paris, 1969.

\bibitem{conrad-grothendieckduality}
B.~Conrad.
\newblock {\em Grothendieck duality and base change}, volume 1750 of {\em
  Lecture Notes in Mathematics}.
\newblock Springer-Verlag, Berlin, 2000.

\bibitem{CremonaFisherStoll-minimisationreductioncoveringselliptic}
J.~E. Cremona, T.~A. Fisher, and M.~Stoll.
\newblock Minimisation and reduction of 2-, 3- and 4-coverings of elliptic
  curves.
\newblock {\em Algebra Number Theory}, 4(6):763--820, 2010.

\bibitem{Davenport-onaresultofLipschitz}
H.~Davenport.
\newblock On a principle of {L}ipschitz.
\newblock {\em J. London Math. Soc.}, 26:179--183, 1951.

\bibitem{Har66}
R.~Hartshorne.
\newblock {\em Residues and duality}.
\newblock Lecture Notes in Mathematics, No. 20. Springer-Verlag, Berlin-New
  York, 1966.
\newblock Lecture notes of a seminar on the work of A. Grothendieck, given at
  Harvard 1963/64, With an appendix by P. Deligne.

\bibitem{Har92}
R.~Hartshorne.
\newblock Generalized divisors on {G}orenstein schemes.
\newblock In {\em Proceedings of {C}onference on {A}lgebraic {G}eometry and
  {R}ing {T}heory in honor of {M}ichael {A}rtin, {P}art {III} ({A}ntwerp,
  1992)}, volume~8, pages 287--339, 1994.

\bibitem{laga-f4paper}
J.~Laga.
\newblock Arithmetic statistics of {P}rym surfaces.
\newblock {\em Math. Ann.}, 386(1-2):247--327, 2023.

\bibitem{lagathorne2024smallheightoddhyperelliptic}
J.~Laga and J.~Thorne.
\newblock 100\% of odd hyperelliptic {J}acobians have no rational points of
  small height.
\newblock arXiv preprint \texttt{2405.10224v1}, 2024.

\bibitem{Nak89}
J.~Nakagawa.
\newblock Binary forms and orders of algebraic number fields.
\newblock {\em Invent. Math.}, 97(2):219--235, 1989.

\bibitem{PlatonovRapinchuk-Alggroupsandnumbertheory}
V.~Platonov and A.~Rapinchuk.
\newblock {\em Algebraic groups and number theory}, volume 139 of {\em Pure and
  Applied Mathematics}.
\newblock Academic Press, Inc., Boston, MA, 1994.
\newblock Translated from the 1991 Russian original by Rachel Rowen.

\bibitem{PoonenRainsRandom}
B.~Poonen and E.~Rains.
\newblock Random maximal isotropic subspaces and {S}elmer groups.
\newblock {\em J. Amer. Math. Soc.}, 25(1):245--269, 2012.

\bibitem{PoonenStoll-localglobaldensities}
B.~Poonen and M.~Stoll.
\newblock A local-global principle for densities.
\newblock In {\em Topics in number theory ({U}niversity {P}ark, {PA}, 1997)},
  volume 467 of {\em Math. Appl.}, pages 241--244. Kluwer Acad. Publ.,
  Dordrecht, 1999.

\bibitem{Silverman-arithmeticellcurvesbook}
J.~H. Silverman.
\newblock {\em The arithmetic of elliptic curves}, volume 106 of {\em Graduate
  Texts in Mathematics}.
\newblock Springer, Dordrecht, second edition, 2009.

\bibitem{soundararajan-equidistributionzeroespolys}
K.~Soundararajan.
\newblock Equidistribution of zeros of polynomials.
\newblock {\em Amer. Math. Monthly}, 126(3):226--236, 2019.

\bibitem{StollCremona-reductionbinaryforms}
M.~Stoll and J.~E. Cremona.
\newblock On the reduction theory of binary forms.
\newblock {\em J. Reine Angew. Math.}, 565:79--99, 2003.

\bibitem{Swaminathan-thesis}
A.~A. Swaminathan.
\newblock 2-{S}elmer groups, 2-class groups, and the arithmetic of binary
  forms.
\newblock Available at
  \url{http://arks.princeton.edu/ark:/88435/dsp012v23vx55w}, 2022.

\bibitem{stacks-project}
{The Stacks project authors}.
\newblock The {S}tacks project.
\newblock \url{https://stacks.math.columbia.edu}, 2024.

\bibitem{Tho14}
J.~A. Thorne.
\newblock A remark on the arithmetic invariant theory of hyperelliptic curves.
\newblock {\em Math. Res. Lett.}, 21(6):1451--1464, 2014.

\bibitem{Thorne-reduction}
J.~A. Thorne.
\newblock Reduction theory for stably graded {L}ie algebras, 2023.
\newblock Preprint.

\bibitem{Woo11}
M.~M. Wood.
\newblock Rings and ideals parameterized by binary {$n$}-ic forms.
\newblock {\em J. Lond. Math. Soc. (2)}, 83(1):208--231, 2011.

\bibitem{Woo14}
M.~M. Wood.
\newblock Parametrization of ideal classes in rings associated to binary forms.
\newblock {\em J. Reine Angew. Math.}, 689:169--199, 2014.

\end{thebibliography}

\end{document}